%% file: DRcomb.tex
\documentclass[11pt,a4paper]{article}
\usepackage{amsmath}
\usepackage{amsfonts}
\usepackage{amssymb}
\usepackage{amsthm}
\usepackage[table]{xcolor}
\usepackage[colorlinks=true,linkcolor=blue,citecolor=blue]{hyperref}
\usepackage{enumerate}
\usepackage{graphicx}
\usepackage{sudoku}
\usepackage{tikz}
\usetikzlibrary{arrows}
\usepackage{caption}
\usepackage{subcaption}
\usepackage{geometry,rotating}

\newtheorem{theorem}{Theorem}[section]
\newtheorem{proposition}{Proposition}[section]
\newtheorem{corollary}{Corollary}[section]
\newtheorem{alg}{Algorithm}[section]
\theoremstyle{remark}
\newtheorem{remark}{Remark}[section]
\newtheorem{example}{Example}[section]

\newcommand{\Fix}{\operatorname{Fix}}
\newcommand{\vect}{\operatorname{vec}}

\newcommand{\qede}{\hspace*{\fill}$\Diamond$\medskip}

\newcommand{\R}{\mathbb{R}}
\newcommand{\C}{\mathcal{C}}
\newcommand{\cb}{\cellcolor{black}}

\title{Recent Results on Douglas--Rachford Methods for Combinatorial Optimization Problems\footnote{All authors are at CARMA, University of Newcastle, Callaghan, NSW 2308, Australia.}}
\author{Francisco J. Arag\'on Artacho\thanks{Email: \url{francisco.aragon@ua.es}}
  \and Jonathan M. Borwein\thanks{Also Distinguished Professor, KAU Jeddah, SA. Email: \url{jon.borwein@gmail.com}}
  \and Matthew K. Tam\thanks{Email: \url{matthew.k.tam@gmail.com}}}

\begin{document}

\maketitle

\begin{abstract} We discuss recent positive experiences applying convex feasibility algorithms of Douglas--Rachford type to highly combinatorial and far from convex problems.
\end{abstract}

\section{Introduction}  Douglas--Rachford iterations, as defined in Section~\ref{sec:conv}, are moderately well understood when applied to finding a point in the intersection of two convex sets.  Over the past decade, they have proven very effective in some highly non-convex settings; even more surprisingly this is the case for some highly discrete problems.  In this paper we wish to advertise the use of Douglas--Rachford methods in such combinatorial settings.  The remainder of the paper is organized as follows.

In Section~\ref{sec:conv}, we recapitulate what is proven in the convex setting. This is followed, in Section~\ref{sec:prod}, by a review of the normal way of handling a (large) finite number of sets in the product space. In Section~\ref{sec:nonconv}, we reprise what is known  in the non-convex setting. Now there is less theory but significant and often positive experience. In Section~\ref{sec:app}, we turn to more detailed discussions of combinatorial applications before focusing, in Section~\ref{sec:sudoku}, on solving \emph{Sudoku puzzles}, and, in Section~\ref{sec:vis}, on solving \emph{Nonograms}. It is worth noting that both of these are NP-complete as decision problems. We complete the paper with various concluding remarks in Section~\ref{sec:conc}.

\section{Convex Douglas--Rachford methods}\label{sec:conv}
In this section we review what is known about the behaviour of Douglas--Rachford methods applied to a finite family of closed and convex sets.

\subsection{The classical Douglas--Rachford method}
The classical Douglas--Rachford scheme was originally introduced in connection with partial differential equations arising in heat conduction \cite{douglas1956numerical}, and convergence later proven as part of \cite{lions1979splitting}. Given two subsets $A,B$ of a Hilbert space, $\mathcal H$, the scheme iterates by repeatedly applying the \emph{$2$-set Douglas--Rachford operator},
 $$T_{A,B}:=\frac{I+R_BR_A}{2},$$
where $I$ denotes the \emph{identity mapping}, and $R_A(x)$ denotes the \emph{reflection} of a point $x\in \mathcal H$
in the set $A$. The reflection can be defined as
$$R_A(x):=2P_A(x)-x,$$
where $P_A(x)$ is the \emph{closest point projection} of the point $x$ onto the set $A$, that is,
$$P_A(x):=\left\{z\in A\colon \|x-z\|=\inf_{a\in A}\|x-a\|\right\}.$$
In general, the projection $P_A$ is a set-valued mapping. If $A$ is closed and convex, the projection is uniquely defined for every point in $\mathcal H$,  thus yielding a single-valued mapping (see e.g.~\cite[Th.~4.5.1]{BZ05}).

In the literature, the Douglas--Rachford scheme is also known as ``\emph{reflect--reflect--average}" \cite{borwein2011douglas}, and ``\emph{averaged alternating reflections (AAR)}" \cite{bauschke2004finding}.

Applied to closed and convex sets, convergence is well understood and can be explained by using the theory of (firmly) nonexpansive mappings.

\begin{theorem}[Douglas--Rachford, Lions--Mercier] \label{th:DR}
 Let $A,B\subseteq\mathcal H$ be closed and convex with nonempty intersection. For any $x_0\in\mathcal H$, set $x_{n+1}=T_{A,B}x_n$. Then $(x_n)$ converges weakly to a point $x$ such that $P_Ax\in A\cap B$.
\end{theorem}

As part of their analysis of \emph{von Neumann's alternating projection method}, Bauschke and Borwein \cite{bauschke1993convergence} introduced the notion of the \emph{displacement vector}, $v$, and used the sets $E$ and $F$ to generalize $A\cap B$.
\begin{equation*}
 v:=P_{\overline{B-A}}(0),\quad E:=A\cap(B-v),\quad F:=(A+v)\cap B.
\end{equation*}
Note, if $A\cap B\neq\emptyset$ then $E=F=A\cap B$.

The same framework was utilized by Bauschke, Combettes and Luke \cite{bauschke2004finding} to analyze the Douglas--Rachford method.

  \begin{theorem}[Infeasible case {\cite[Th.~3.13]{bauschke2004finding}}] \label{th:DR2}
 Let $A,B\subseteq\mathcal H$ be closed and convex. For any $x_0\in\mathcal H$, set $x_{n+1}=T_{A,B}x_n$.
 Then the following hold.
 \begin{enumerate}[(i)]
  \item $x_{n+1}-x_n=P_BR_Ax_n-P_Ax_n\to v$ and $P_BP_Ax_n-P_Ax_n\to v$.
  \item If $A\cap B\neq\emptyset$ then $(x_n)$ converges weakly to a point in $$\Fix(T_{A,B})=(A\cap B)+N_{\overline{A-B}}(0);$$ otherwise, $\|x_n\|\to+\infty$.
  \item Exactly one of the following two alternatives holds.
   \begin{enumerate}[(a)]
    \item $E=\emptyset$, $\|P_Ax_n\|\to+\infty$, and $\|P_BP_Ax_n\|\to+\infty$.
    \item $E\neq\emptyset$, the sequences $(P_Ax_n)$ and $(P_BP_Ax_n)$ are bounded, and their weak cluster points belong to $E$ and $F$, respectively; in fact, the weak cluster points of
     \begin{equation}
      ((P_Ax_n,P_BR_Ax_n)) \text{ and }  ((P_Ax_n,P_BP_Ax_n)) \label{eq:bestapprox}
     \end{equation}
    are best approximation pairs relative to $(A,B)$.
   \end{enumerate}
 \end{enumerate}
 \end{theorem}

Here, $N_C(x):=\{u\in\mathcal H:\langle c-x,u\rangle\leq 0,\forall c\in C\}$ denotes the \emph{normal cone} to a convex set $C\subset\mathcal H$ at a point $x\in C$, and $\Fix (T):=\{x\in\mathcal H:x\in T(x)\}$ denotes the set of \emph{fixed points} of the mapping $T$.


\begin{remark}[Behaviour of best approximation pairs]
 If best approximation pairs relative to $(A,B)$ exist and $P_A$ is weakly continuous, then the sequences in (\ref{eq:bestapprox}) actually converge weakly to such a pair \cite[Remark~3.14(ii)]{bauschke2004finding}.

 Since $x_n/n\to -v$, $\|x_n/n\|$ can be used to approximate $\|v\|=d(A,B)$ \cite[Remark~3.16(ii)]{bauschke2004finding}. \qede
\end{remark}

We turn next to an alternative new method:

\subsection{The cyclic Douglas--Rachford method}\label{sssec:cdr}
There are many possible generalizations of the classic Douglas--Rachford iteration. Given three sets $A,B,C$ and $x_0\in\mathcal H$, an obvious candidate is the iteration defined by repeatedly setting $x_{n+1}:=T_{A,B,C}x_n$ where
 \begin{equation}
 T_{A,B,C}:=\frac{I+R_CR_BR_A}{2}. \label{eq:3setdr}
 \end{equation}

For closed and convex sets, like $T_{A,B}$, the mapping $T_{A,B,C}$ is firmly nonexpansive, and has at least one fixed point provided $A\cap B\cap C\neq\emptyset$. Using a well known theorem of Opial \cite[Th.~1]{opial67}, $(x_n)$ can be shown to converge weakly to a fixed point. However, attempts to obtain a point in the intersection using said fixed point have, so far, been unsuccessful.

\begin{example}[Failure of three set Douglas--Rachford iterations.] \label{ex:3setdr}
 We give an example showing the iteration described in (\ref{eq:3setdr}) can fail to find a feasible point. Consider the one-dimensional subspaces $A,B,C\subset\R^2$ defined by
 \begin{align*}
  A &:= \{\lambda (0,1):\lambda\in\R\}, \\
  B &:= \{\lambda (\sqrt{3},1):\lambda\in\R\}, \\
  C &:= \{\lambda (-\sqrt{3},1):\lambda\in\R\}.
 \end{align*}
 Then $A\cap B\cap C=\{(0,0)\}$.

 Let $x_0=(-\sqrt{3},-1)$. Since $x_0\in\Fix R_CR_BR_A$,
  $$x_0\in\Fix\frac{I+R_CR_BR_A}{2}.$$
 However,
  $$P_Ax_0=(0,-1),\quad P_Bx_0=x_0=(-\sqrt{3},-1),\qquad P_Cx_0=(-\sqrt{3}/2,1/2).$$
 That is, $P_Ax_0,P_Bx_0,P_Cx_0\not\in A\cap B\cap C$. The trajectory is illustrated in Figure~\ref{fig:3setdr}.
 \qede

 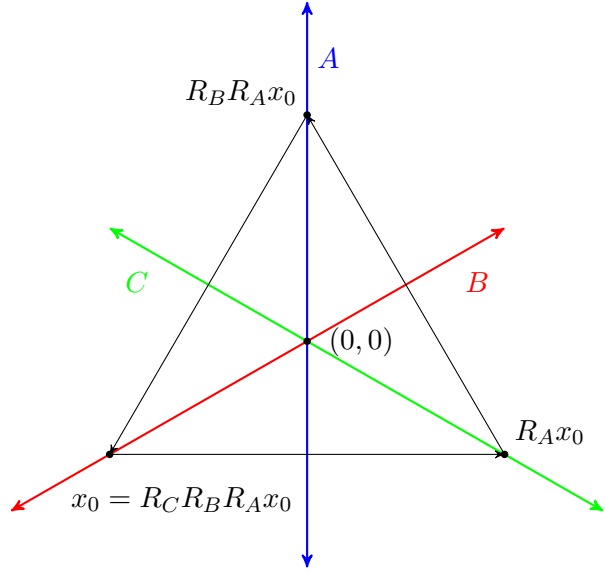
\begin{figure*}
 \begin{center}
  \input{3setdr-v2.tex}
 \caption{Trajectory of Example~\ref{ex:3setdr}.}\label{fig:3setdr}
 \end{center}
 \end{figure*}

\end{example}

Instead, Borwein and Tam \cite{cycdr} considered cyclic applications of $2$-set Douglas--Rachford operators. Given $N$ sets $C_1,C_2,\dots,C_N$, and $x_0\in\mathcal H$, their \emph{cyclic Douglas--Rachford scheme} iterates by repeatedly setting $x_{n+1}:=T_{[C_1,C_2,\dots,C_N]}x_n$, where $T_{[C_1,C_2,\dots,C_N]}$ denotes the \emph{cyclic Douglas--Rachford operator} defined by
 $$T_{[C_1,C_2,\dots,C_N]}:=T_{C_N,C_1}T_{C_{N-1},C_N}\dots,T_{C_2,C_3}T_{C_1,C_2}.$$

In the consistent case, the iterations behave analogously to the classical Douglas--Rachford scheme (cf. Theorem~\ref{th:DR}).
\begin{theorem}[Cyclic Douglas--Rachford]\label{th:cycDR}
  Let $C_1,C_2,\dots,C_N\subseteq\mathcal H$ be closed and convex sets with a nonempty intersection. For any $x_0\in\mathcal H$, set $x_{n+1}=T_{[C_1\,C_2\,\dots\,C_N]}x_n$. Then $(x_n)$ converges weakly to a point $x$ such that $P_{C_i}x = P_{C_j}x$, for all indices $i,j$. Moreover, $P_{C_j}x\in \bigcap_{i=1}^NC_i$, for each index $j$.
\end{theorem}

 \begin{example}[Example~\ref{ex:3setdr} revisited]\label{ex:3setcycdr}
   Consider the cyclic Douglas--Rachford scheme applied to the sets of Example~\ref{ex:3setdr}. As before, let $x_0=(-\sqrt{3},-1)$. By Theorem~\ref{th:cycDR}, the sequence $(x_n)$ converges to a point $x$ such that
    $$P_Ax=P_Bx=P_Cx=(0,0).$$
 Furthermore, $P_A,P_B,P_C$ are orthogonal projections, hence $x=(0,0)$. The trajectory is illustrated in Figure~\ref{fig:3setcycdr}.

As a consequence of the problem's rotational symmetry, the sequence of Douglas--Rachford operators can be described by
$$T_{A,B} x_n = P_C x_n,\quad T_{B,C} T_{A,B} x_n = P_A P_C x_n,\quad x_{n+1}=T_{C,A} T_{B,C} T_{A,B} x_n = P_B P_A P_C x_n.$$
That is, starting at $x_0$, the cyclic Douglas--Rachford trajectory applied to the $A,B,C$, coincides with von Neumann's alternating projection method applied to $C,A,B$ (cf. \cite[Cor.~3.1]{cycdr}).\qede

   \begin{figure*}
     \begin{center}
       \input{cycdr-ex.tex}
       \caption{Trajectory of Example~\ref{ex:3setcycdr}. Solid black arrows represent $2$-set Douglas--Rachford iterations (i.e. they connect the sequence $x_0,\allowbreak T_{A,B}x_0, \allowbreak T_{B,C}T_{A,B}x_0, \allowbreak T_{C,A}T_{B,C}T_{A,B}x_0,\dots$). Constructions (reflect-reflect-average) are dotted.}\label{fig:3setcycdr}
     \end{center}
   \end{figure*}
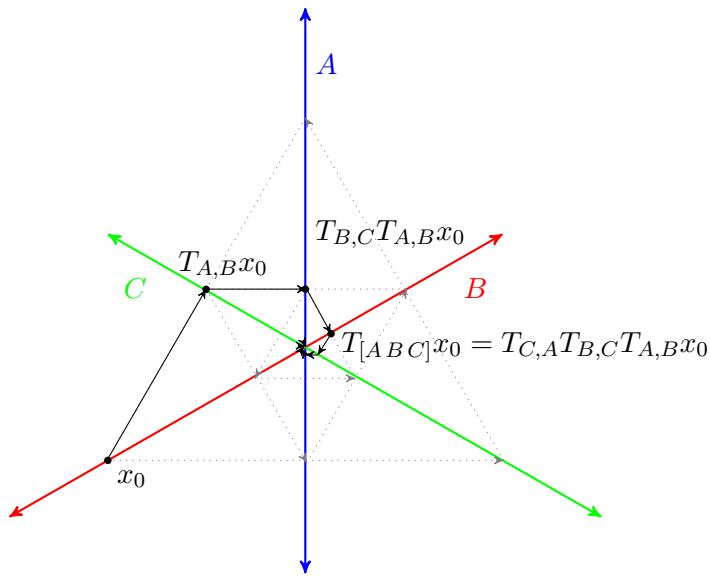

 \end{example}

If $N=2$ and $C_1\cap C_2=\emptyset$ (the inconsistent case), unlike the classical Douglas--Rachford scheme, the iterates are not unbounded (cf. Theorem~\ref{th:DR2}). Moreover, there is evidence to suggest that the scheme can be used to produce best approximation pairs relative to $(C_1,C_2)$ whenever they exist.

The framework of Borwein and Tam \cite{cycdr}, can also be used to derive a number of applicable variants. A particularly nice one is the \emph{averaged Douglas--Rachford scheme} which, for any $x_0\in\mathcal H$, iterates by repeatedly setting\footnote{Here indices are understood modulo $N$. That is, $C_{N+1}:=C_1$.}
 $$x_{n+1} := \frac{1}{N}\left(\sum_{i=1}^NT_{C_i,C_{i+1}}\right)x_n.$$
Since each $2$-set Douglas--Rachford operator can be computed independently the iteration easily parallelizes.

\begin{remark}[Failure of norm convergence] It is known that the alternating projection method  may fail to converge in norm \cite{bmr04}, and it follows that both classical and cyclic Douglas-Rachford methods may also only converge weakly. For the classical method this may be deduced from  \cite[Section 5]{bmr04}. For the cyclic case, see  \cite[Cor. 3.1.]{cycdr} for details. \qede \end{remark}

\subsubsection{Numerical Performance}
Applied to the problem of finding a point in the intersection of $N$ balls in $\R^n$, initial numerical experiments suggest that the cyclic Douglas--Rachford outperforms the classical Douglas--Rachford scheme \cite{cycdr}.

To ensure this performance is not an artefact of having highly symmetrical constraints, the same problem, replacing the balls with prolate spheroids (the type of ellipsoid obtained by rotating a $2$-dimensional ellipse around its major axis) having one common focus was considered. Unlike ball constraints, there is no simple formula for computing the projection onto a spheroid. However, the projections can be computed efficiently.  The process reduces to numerically solving, for $t$, the equation
 $$\frac{a^2u^2}{(a^2-t)^2}+\frac{b^2v^2}{(b^2-t)^2} = 1,$$
for constants $a,b>0$ and $u,v\in\R$. For further details, see \cite[Ex.~2.3.18]{borwein2010convex}.

In the spheroid case, the computational results are very similar to the ball case, considered in~\cite{cycdr}. An example having three spheroids in $\R^2$ is illustrated in Figure~\ref{fig:cycDRellipse}.

\begin{figure*}
  \begin{center}
    \includegraphics[width=0.7\textwidth]{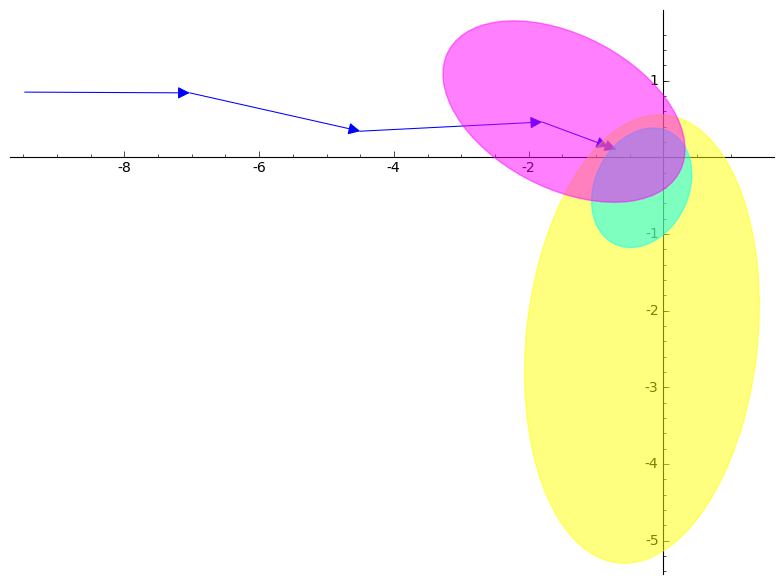}
    \caption{A cyclic Douglas--Rachford trajectory for three ellipses in $\R^2$. Blue arrows represent $2$-set Douglas--Rachford iterations (i.e. they connect the sequence $x_0,\allowbreak T_{A,B}x_0, \allowbreak T_{B,C}T_{A,B}x_0, \allowbreak T_{C,A}T_{B,C}T_{A,B}x_0,\dots$).}\label{fig:cycDRellipse}
  \end{center}
\end{figure*}

\section{Feasibility problems in the product space}\label{sec:prod}

Given $C_1,C_2,\cdots,C_N\subset\R^n$, the \emph{feasibility problem}\footnote{In this context, ``feasibility" and ``satisfiability" can be used interchangeably.} asks:
 \begin{equation}\label{eq:feasibilityproblem}
  \text{Find }x\in\bigcap_{i=1}^NC_i\subset\R^n.
 \end{equation}

A great many  optimization and reconstruction problems, both continuous and combinatorial, can be cast within this framework.

Define two sets $C,D\subset(\R^n)^N$ by
 \begin{equation*}
  C:=\prod_{i=1}^NC_i,\quad D:=\{(x,x,\dots,x)\in(\R^n)^N:x\in\R^n\}.
 \end{equation*}
While the set $D$, the \emph{diagonal}, is always a closed subspace, the properties of $C$ are largely inherited. For instance, when $C_1,C_2,\dots,C_N$ are closed and convex, so is $C$.

Consider, now, the equivalent feasibility problem:
 \begin{equation}\label{eq:productfeasibilityproblem}
  \text{Find }\mathbf x\in C\cap D\subset(\R^n)^N.
 \end{equation}
Equivalent in the sense that
 $$x\in\bigcap_{i=1}^N C_i\iff (x,x,\dots,x)\in C\cap D.$$
Moreover, knowing the projections onto $C_1,C_2,\dots,C_N$, the projections onto $C$ and $D$ can be easily computed. The proof  has recourse to the standard characterization of orthogonal projections, $$p=P_Dx\iff \langle x-p,m\rangle=0\text{ for all }m\in D.$$

\begin{proposition}[Product projections]
For any $\mathbf x=(\mathbf x_1,\ldots,\mathbf x_N)\in(\R^n)^N$ one has
\begin{equation}
 P_C\mathbf x = \prod_{i=1}^NP_{C_i}(\mathbf x_i),\quad P_D\mathbf x=\left(\frac{1}{N}\sum_{i=1}^N\mathbf x_i,\ldots,\frac{1}{N}\sum_{i=1}^N\mathbf x_i\right).
\end{equation}
\end{proposition}
\begin{proof}
 For any $\mathbf c=(\mathbf c_1,\ldots,\mathbf c_N)\in C$,
 $$\|\mathbf x-\mathbf c\|^2 = \sum_{i=1}^N\|\mathbf x_i-\mathbf c_i\|^2\geq\sum_{i=1}^N\|\mathbf x_i-P_C\mathbf x_i\|^2=\left\|\mathbf x-\prod_{i=1}^NP_{C_i}(\mathbf x_i)\right\|^2.$$
 This proves the form of the projection onto $C$. Let $(\mathbf p,\ldots,\mathbf p)\in D$ be the projection of $\mathbf x$ onto $D$. For any $\mathbf m\in\R^n$, one has $(\mathbf m,\ldots,\mathbf m)\in D$, and then
 $$0=\langle \mathbf x-(\mathbf p,\ldots,\mathbf p),(\mathbf m,\ldots,\mathbf m)\rangle=\sum_{i=1}^N\langle\mathbf x_i-\mathbf p,\mathbf m\rangle=\langle\sum_{i=1}^N\mathbf x_i-N\mathbf p,\mathbf m\rangle;$$
 whence,
 $\mathbf p=\frac{1}{N}\sum_{i=1}^N\mathbf x_i$,
%
 and the proof is complete.
\end{proof}

Most projection algorithms can be applied to feasibility problems with any finite number of sets without significant modification. An exception is the Douglas--Rachford scheme, which  until \cite{cycdr} had only been successfully investigated for the case of two sets. This has made the product formulation crucial for the Douglas--Rachford scheme.

 \section{Non-convex Douglas--Rachford methods} \label{sec:nonconv}

 While there is not nearly so much theory in the non-convex setting, there are some useful beginnings:

     \begin{figure*}
 \begin{center}
  \includegraphics[width=0.5\textwidth]{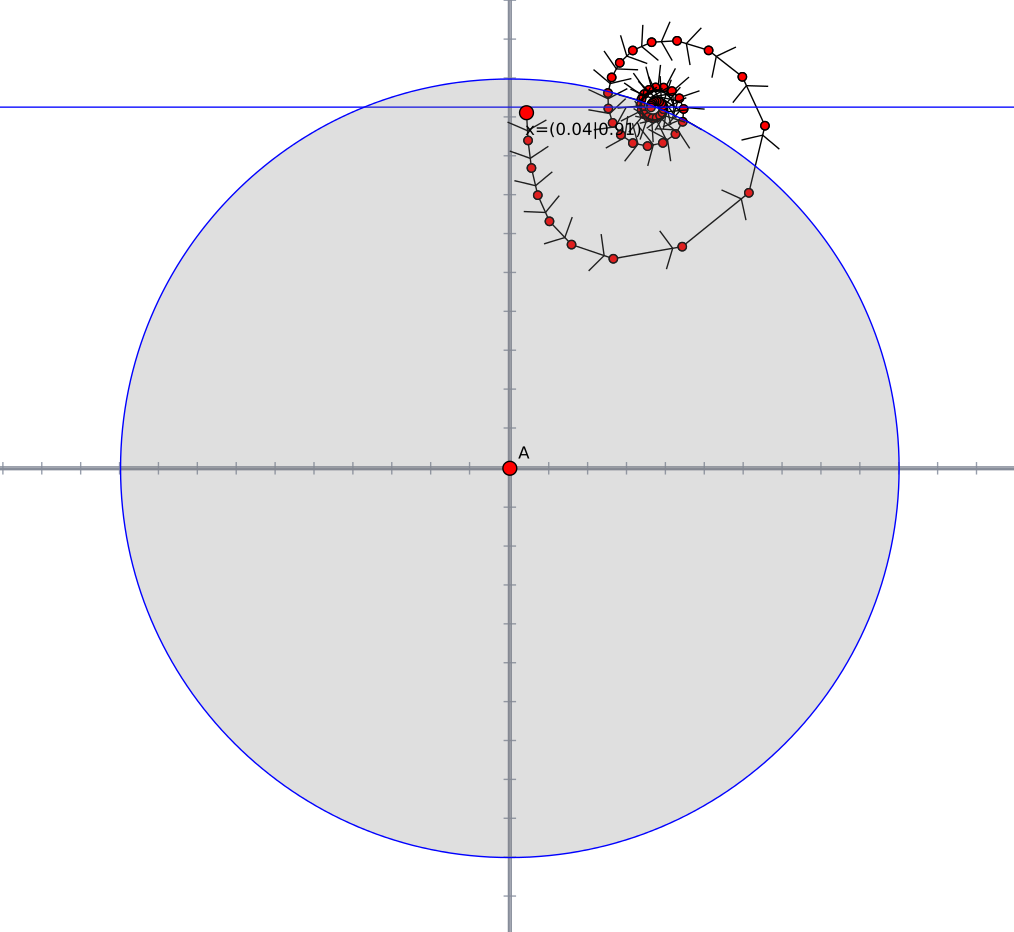}
    \caption{A Douglas--Rachford trajectory showing local convergence to a feasible point, as in Theorem~\ref{th:borweinDR}, exhibiting ``spiralling" behaviour.}
 \end{center}
\end{figure*}

 \subsection{Theoretical underpinnings}

  As a prototypical non-convex scenario, Borwein and Sims \cite{borwein2011douglas} considered
   the Douglas--Rachford scheme applied to a Euclidean sphere and a line. More precisely, they looked at the sets
   $$S:=\{x\in\R^n:\|x\|=1\},\quad L:=\{\lambda a+\alpha b\in\R^n:\lambda\in\R\},$$
  where, without loss of generality, $\|a\|=\|b\|=1,a\perp b,\alpha>0$.
We summarize their findings.

 Appropriately normalized the iteration becomes
  \begin{equation}\label{iter2}\begin{array}{l}
x_{n+1}(1)=x_n(1)/\rho_n,\\
x_{n+1}(2)=\alpha+(1-1/\rho_n)x_n(2), \text{ and}\\
x_{n+1}(k)=(1-1/\rho_n)x_n(k),\text { for } k=3,\ldots,N,
\end{array}
\end{equation}
where $\rho_n:=\|x_n\|:=\sqrt{x_n(1)^2+\ldots+x_n(N)^2}$, see \cite{borwein2011douglas} for details.
  The non-convex sphere, $S$, provides an accessible model of many reconstruction problems in which the magnitude, but not the phase, of a signal is measured.

Note $\alpha\in[0,1]$ represents the consistent case, and $\alpha>1$ the inconsistent one.

  \begin{theorem}[Sphere and line]\label{th:borweinDR}
    Given $x_0\in\R^n$ define $x_{n+1}:=T_{S,L}x_n$. Then:
   \begin{enumerate}
    \item If $0<\alpha<1$, $(x_n)$ is locally convergent at each of $\pm\sqrt{1-\alpha^2}a+\alpha b$.
    \item If $\alpha=0$ and $x_0(1)>0$, $(x_n)$ converges to $a$.
    \item If $\alpha=1$ and $x_0(1)\neq0$, $(x_n)$ converges to $\hat yb$ for some $\hat y>1$.
    \item If $\alpha>1$ and $x_0(1)\neq0$, $\|x_n\|\to\infty$.
   \end{enumerate}
  \end{theorem}

\begin{figure*}[t]
 \begin{center}
  \includegraphics[width=0.4\textwidth]{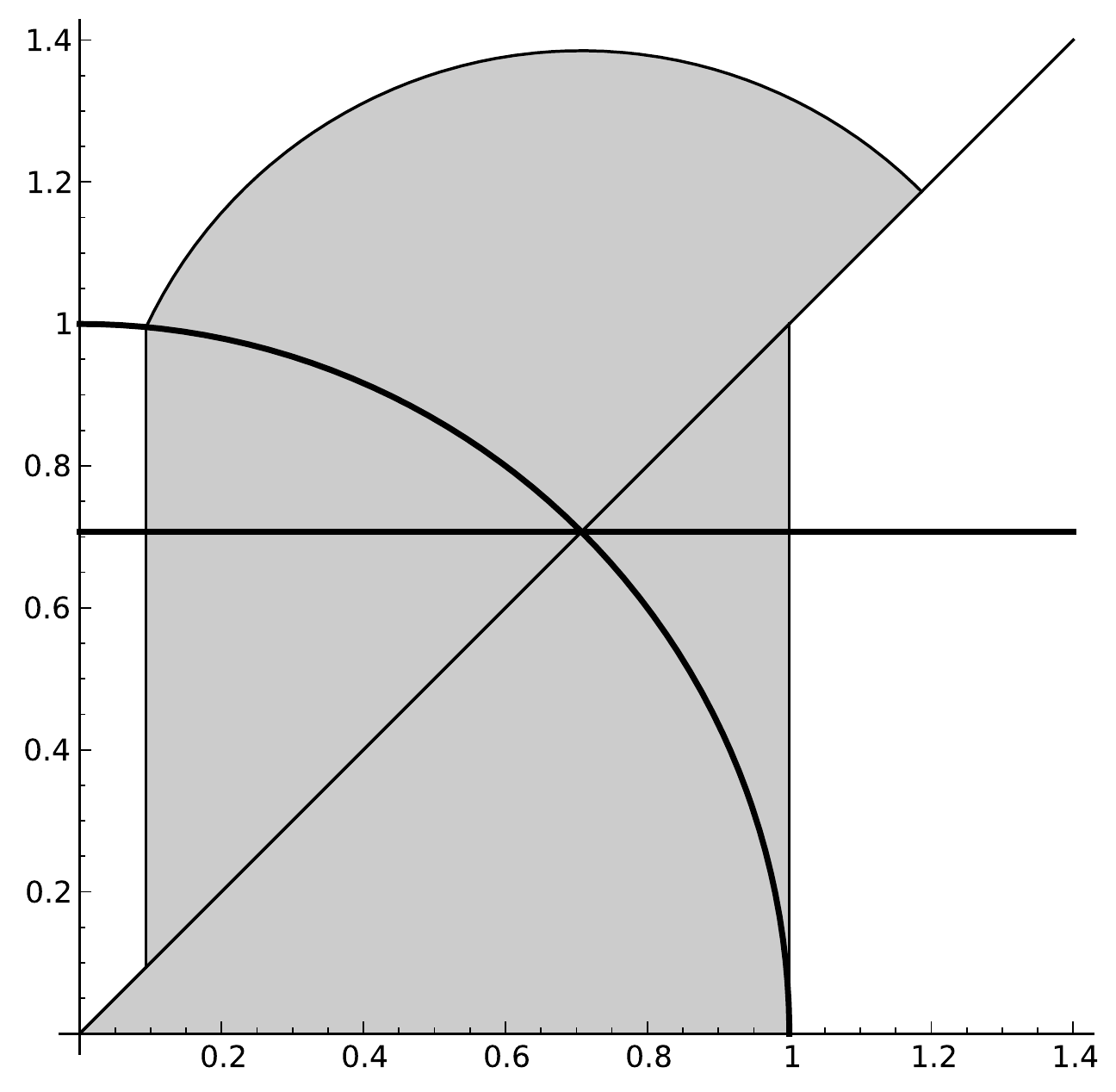}
   \caption{The explicit region of convergence (grey) given in~\cite{aragon2012global}.} \label{fig:fran-cvgt}
 \end{center}
\end{figure*}

  Replacing $L$ with the proper affine subspace, $A:=A_0+\alpha b$ for some non-trivial subspace $A_0$, $(x_n)$ needs to be excluded from $A_0^\perp$. Now, if $x_0 \not\in A_0^\perp$ then for some infeasible $q\neq 0$, $x_0\in Q:=A_0^\perp +\R q$, then $(x_n)$ are confined to the subspace $Q$. Theorem~\ref{th:borweinDR} can, with some care then be extended to the following.

    \begin{figure*}
 \begin{center}
  \includegraphics[width=0.4\textwidth]{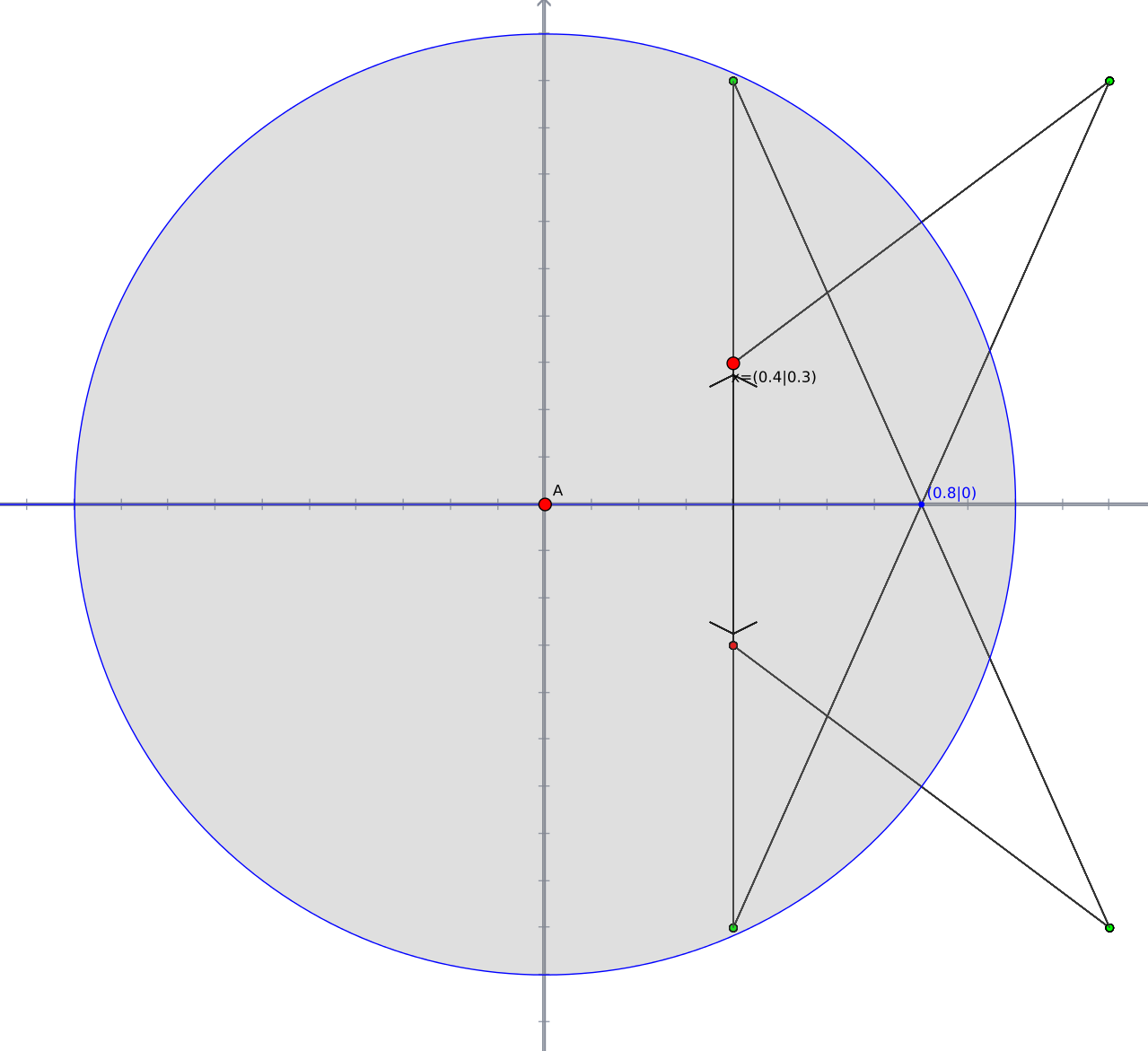}
    \caption{A two cycle $(2/5,\pm 3/10)$.}\label{fig:2cyc}
 \end{center}
\end{figure*}

  \begin{corollary}[Sphere and non-trivial affine subspace]\label{cor:DR}
   For each feasible point $p\in S\cap A\cap Q$ there exists a neighbourhood $N_p$ of $p$ in $Q$ such that starting from any $x_0\in N_p$ the Douglas--Rachford scheme converges to $p$.
  \end{corollary}

  If in Theorem~\ref{th:borweinDR} $x_0(1)=0$, the behaviour of the scheme can  provably be quite chaotic \cite{borwein2011douglas}. Indeed, this was a difficulty encountered by Arag\'on and Borwein \cite{aragon2012global}, in giving an explicit region of convergence for the $\R^2$ case with $\alpha=1/\sqrt{2}$.

  \begin{theorem}[Global convergence {\cite[Th.~2.1]{aragon2012global}}] \label{th:aragonDR}
   Let $x_0\in[\epsilon,1]\times[0,1]$ with $\epsilon:=(1-2^{-1/3})^{3/2}\approx 0.0937$. Then the sequence generated by the Douglas--Rachford scheme of \eqref{iter2} with starting point $x_0$ is convergent to $(1/\sqrt{2},1/\sqrt{2})$.
  \end{theorem}

The restriction to $\alpha=1/\sqrt{2}$ was largely made for notational simplicity.

In fact, a careful analysis show that the region of convergence is actually larger \cite[Remark~2.12]{aragon2012global}, as illustrated in Figure~\ref{fig:fran-cvgt}.

  \begin{example}[Failure of Douglas--Rachford for a half-line and circle]
  Just replacing a line by a half line in the setting of Borwein--Sims \cite{borwein2011douglas,aragon2012global} is enough to allow complicated periodic behaviour.

Let
 $$A:=S_{\R^2}:=\{x\in\R^2:\|x\|=1\},\quad B:=\{(x_1,0)\in\R^2:x_1\leq a\}.$$
Then
 $$P_Ax = \left\{\begin{array}{ll}
            x/\|x\| & \text{if }x\neq 0, \\
            A & \text{otherwise.} \\
          \end{array}\right.,\quad
   P_Bx = \left\{\begin{array}{ll}
            (x_1,0) & \text{if }x_1\leq a\\
            (a,0)   & \text{otherwise.}
          \end{array}\right.$$
The following holds.

\begin{proposition}
 For each $a\in(0,1)$, there is a $2$-cycle starting at
  $$x_0=\left(a/2,\sqrt{1-a^2}/2\right).$$
\end{proposition}
\begin{proof}
 Since $\|x_0\|=\frac{1}{2}$,
  $$R_Ax_0 = 2\frac{x_0}{\|x_0\|}-x_0=3x_0.$$
 Since $(R_Ax_0)_1=3a/2>a$, $P_BR_Ax=(a,0)$ and hence
  $$T_{A,B}x_0 = \frac{x_0+2(a,0)-3x_0}{2}=(a,0)-x_0 = \left(a/2,-\sqrt{1-a^2}/2\right).$$
By symmetry, $T_{A,B}^2x_0=x_0$.
\end{proof}

If we replace $B$ by the singleton $\{(a,0)\}$ or the doubleton $\{(a,0),(-1,0)\}$ we obtain the same two-cycle.
The case of a singleton shows the need for $A$ to be non-trivial in Corollary \ref{cor:DR}.

This cycle is illustrated in Figure \ref{fig:2cyc} for $a=4/5$ which leads to a rational cycle. For points near the cycle, the iteration generates remarkably subtle limit cycles as shown in  Figure~\ref{fig:orbit}.\footnote{See \url{http://carma.newcastle.edu.au/DRmethods/comb-opt/2cycle.html} for an animated version.}
 \qede
\end{example}

 \begin{figure*}
 \begin{center}
  \includegraphics[width=0.4\textwidth]{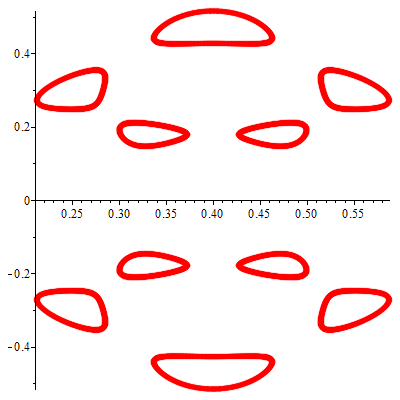}
    \caption{The orbit starting at $(.49,.21)$.}\label{fig:orbit}
 \end{center}
\end{figure*}

\bigskip

  In \cite{hesse2012nonconvex}, Hesse and Luke utilize  \emph{$(S,\epsilon)$-(firm) nonexpansiveness}, a relaxed local version of (firm) nonexpansiveness, a notion which quantifies how ``close" to being (firmly) nonexpansive a mapping is. Together with a coercivity condition, and appropriate notions of super-regularity and linear strong regularity, their framework can be utilized to prove local convergence of the Douglas--Rachford scheme, if the first reflection is performed with respect to a subspace, see~\cite[Th.~42]{hesse2012nonconvex}. The order of reflection is reversed, so the results of Hesse and Luke do not directly overlap with that of Arag\'on, Borwein and Sims. This is not a substantive difference.
%
%

\begin{remark}
 Recently Bauschke, Luke, Phan and Wang \cite{mapsparse} obtained local convergence results for a simpler algorithm, \emph{von Neumann's alternating projection method (MAP)}, applied to sparsity optimization with affine constraints --- a form of combinatorial optimization (Sudoku, for example, can be modelled in this framework \cite{babu2010linear}). In practice, however, our experience is that MAP often fails to converge satisfactorily when applied to these problems. \qede
\end{remark}

 \subsection{A summary of applications}\label{ssec:list}

We briefly mention a variety of highly non-convex, primarily combinatorial, problems where some form of Douglas--Rachford algorithm has proven very fruitful.

  \begin{enumerate}
  \item \emph{Protein folding} and \emph{graph coloring} problems were first studied via Douglas--Rachford methods in \cite{elser2006deconstructing} and \cite{elser2007searching}, respectively.
   \item \emph{Image retrieval} and \emph{phase reconstruction} problems are analyzed in some detail in \cite{bauschke2002phase,bauschke2003hybrid}. The \emph{bit retrieval} problem is considered in \cite{elser2007searching}.
   \item The \emph{$N$-queens problem}, which requests the placement of $N$ queens on a $N \times N$ chessboard, is studied and solved in \cite{schaad2010modeling}.
     \item \emph{Boolean satisfiability} is treated in \cite{elser2007searching,gravel2008divide}. Note that the three variable case, \emph{3-SAT}, was the first problem to be shown \emph{NP-complete} \cite{np}.
   \item \emph{TetraVex}\footnote{Also known as \emph{McMahon Squares} in honour of the great English combinatorialist, Percy MacMahon, who examined them nearly a century ago.} is an edge-matching puzzle (see Figure~\ref{fig:tetravex}), whose NP-completeness is discussed in \cite{takenaga2009tetravex}, was studied in \cite{bansal2010code}.\footnote{Pulkit Bansal did this as a 2010 NSERC summer student with Heinz Bauschke and Xianfu Wang.} Problems up to size $4\times 4$ could be solved in an average of $200$ iterations. There are $10^{2n(n+1)}$ base-$10$ $n\times n$ boards, with $n=3$ being the most popular.

\begin{figure*}
 \begin{center}
  \includegraphics[width=0.75\textwidth]{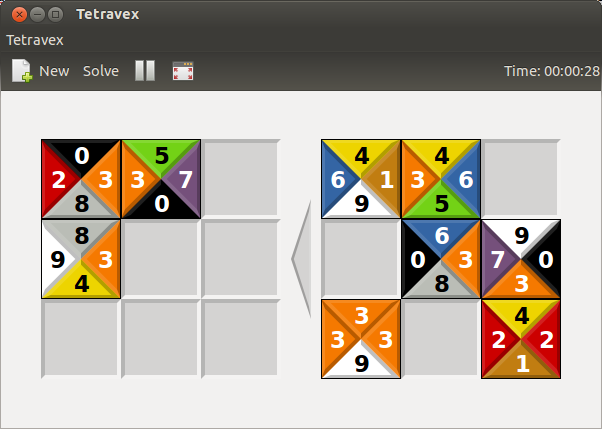}
  \caption{A game of $3\times 3$ TetraVex  being played in \emph{GNOME TetraVex}. Square tiles on the right board must be moved  to the left board so that all touching numbers agree.} \label{fig:tetravex}
 \end{center}
\end{figure*}

   \item Solutions of (very large) \emph{Sudoku puzzles} have been studied in  \cite{schaad2010modeling,elser2007searching}. For a discussion of NP-completeness of determining solvability of Sudokus see \cite{takayuki2003complexity}. The effective solution of Suduko puzzles forms the basis of Section~\ref{sec:sudoku}.
       \item \emph{Nonograms} \cite{nono1,nono2} are a more recent NP-complete Japanese puzzle whose solution by Douglas--Rachford methods is described in Section~\ref{sec:vis}.\footnote{Japanese, being based on ideograms, does not lead itself to anagrams, crosswords or other word puzzles; this in part explains why so many good numeric and combinatoric games originate in Japan.}
  \end{enumerate}

\section{Successful combinatorial applications} \label{sec:app}
 The key to successful application is two-fold.

 First, the iteration must converge---at least with high probability. Our experience is when that happens, random restarts in case of failure are very fruitful.  As we shall show, often this depends on making good decisions about how to model the problem.

  Second, one must be able to compute the requisite projections in closed form---or to approximate them efficiently numerically. As we shall indicate this is frequently possible for significant problems.

  When these two events obtain, we are in the pleasant position of being able to lift much of our experience as continuous optimizers to the combinatorial milieu.

\subsection{Model formulation}\label{ssec:model}
Within the framework of feasibility problems, there can be numerous ways to model a given type of problem. The product space formulation (\ref{eq:productfeasibilityproblem}) gives one example, even without assuming any additional knowledge of the underlying problem.

The chosen formulation heavily influences the performance of projection algorithms. For example, in initial numerical experiments, the cyclic Douglas--Rachford scheme of Section \ref{sssec:cdr}, was directly applied to (\ref{eq:feasibilityproblem}). As a serial algorithm, it seems to outperform the classic Douglas--Rachford scheme, which must instead be applied to in the product space (\ref{eq:productfeasibilityproblem}). For details see \cite{cycdr}.

As a heuristic for problems involving one or more non-convex set, the sensitivity of the  Douglas--Rachford method to the formulation used must be emphasized. In the (continuous) convex setting, the formulation influences performance of the algorithm, while in the combinatorial setting, the formulation determines whether or not the algorithm can successfully and reliably solve the problem at hand. Direct applications to feasibility problems with integer constraints have been largely unsuccessful. On the other hand, many of the successful applications outlined in Section~\ref{ssec:list} use binary formulations.


We now outline the basic idea behind these reformations. If
\begin{equation}\label{eq:reformulation1}
 x\in\{c_1,c_2,\dots,c_n\}\subset\R.
\end{equation}
We reformulate $x$ as a vector $y\in\R^n$. If $x=c_i$, then $y=(y_1,\ldots,y_n)$ is defined by
 $$y_j = \left\{\begin{array}{ll}
          1 & \text{if }j=i, \\
          0 & \text{otherwise.}
         \end{array}\right.$$
With this interpretation (\ref{eq:reformulation1}) is equivalent to:
 $$y\in\{e_1,e_2,\dots,e_n\}\subset\R^n,$$
with $y=e_i$ if and only if $x=c_i$.

Choosing $c_1,c_2,\dots,c_n\in\mathbb Z$ takes care of the integer case.

\subsection{Projection onto the set of permutations of points}

In many situations, in order to apply the Douglas--Rachford iteration, one needs to compute the projection of a point $x=(x_1,\ldots,x_n)\in\R^n$ onto the set of permutations of $n$ given points $c_1,\ldots,c_n\in\R$, 
 a set that will be denoted by $\C$.
 We shall see below that this is the case for the Sudoku puzzle.

 As we show next, the projection can be easily and efficiently computed.
In what follows, given $y\in\mathbb R^n$, we will denote by $[y]$ the vector with the same components permuted in nonincreasing order. We need the following classical rearrangement inequality, see~\cite[Th.~368]{HLP52}.

\begin{theorem}[Hardy--Littlewood--P\'olya]\label{th:HLP}
 Any $x,y\in\R^n$ satisfy
  $$x^Ty \leq [x]^T[y].$$
\end{theorem}

%
%
Fix $x\in\R^n$. Denote by $[\C]_x$ the set of vectors in $\C$ (which therefore have the same components but perhaps permuted) such that $y\in [\C]_x$ if the $i$th largest entry of $y$ has the same index in $y$ as the $i$th largest entry of $x$. As a consequence of Theorem~\ref{th:HLP}, one has the following.

\begin{proposition}[Projections on permutations]\label{prop:permutationprojections}
 Denote by $\C \subset\R^n$ the set of vectors whose entries are all permutations of $c_1,c_2,\dots,c_n\in\R$. Then for any $x\in\R^n$,
  $$P_{\C}x=[\C]_x.$$
\end{proposition}
\begin{proof}
For any $c\in \C$,
\begin{align*}
 \|x-c\|^2
 &= \|x\|^2+\|c\|^2 - 2x^Tc \\
 &= \|[x]\|^2 + \|[c]\|^2- 2x^Tc \\
 &\geq \|[x]\|^2 + \|[c]\|^2- 2[x]^T[c] \\
 &= \|[x]-[c]\|^2\\
 &=\|x-y\|^2,\text{ for }y\in[\C]_x.
\end{align*}
 This completes the proof.
\end{proof}

\begin{remark} \label{remark:01projections}
In particular, taking $c_1=1$ and $c_2=c_3=\dots=c_n=0$ one has
 $$\C=\{e_1,e_2,\dots,e_n\},$$
where $e_i$ denotes the $i$th standard basis vector; whence
 $$P_{\C}(x) = \{e_i:x_i=\max\{x_1,x_2,\dots,x_n\}\}.$$
 A direct proof of this special case is given in \cite[Section~5.9]{schaad2010modeling}. \qede
\end{remark}


\begin{remark}\label{remark:algorithm}
Proposition~\ref{prop:permutationprojections} suggests the following algorithm for computing a projection of $x$ onto $\C$. Since the projection, in general, is not unique, we are content with finding the nearest point, $p$, in the set of projections or some other reasonable surrogate.

For convenience, given a vector $y\in(\R^2)^n$, we denote the projections onto the first and second product coordinates by $Q$ and $S$, respectively. That is, if
 $$y=((x_1,c_2),(x_2,c_2),\dots,(x_n,c_n))\in(\R^2)^n,$$
then
 $$Qy:=(x_1,x_2,\dots,x_n),\qquad Sy:=(c_1,c_2,\dots,c_n).$$
 We can now can now  state the following:

\begin{alg}[Projection]
Input: $x\in\R^n$ and $c_1,c_2,\dots,c_n\in\R$.

\begin{enumerate}

 \item By relabelling if necessary, assume $c_i\leq c_{i+1}$ for each $i$.
 \item Set $y=((x_1,c_2),(x_2,c_2),\dots,(x_n,c_n))\in(\R^2)^n$.
 \item Set $z$ to be a vector with the same components as $y$ permuted such that $Qz$ is in non-increasing order.
 \item Output: $p=Sy$.
\end{enumerate}
\end{alg}

In our experience many projections required in combinatorial settings have this level of simplicity. \qede
\end{remark}

\section{Solving Sudoku puzzles} \label{sec:sudoku}
We now demonstrate the reformulation described in Section~\ref{sec:app} with Sudoku, modelled first as an integer feasibility problem, and secondly as a binary feasibility problem.

We introduce some notation. Denote by $A[i,j]$, the $(i,j)$-th entry of the matrix $A$. Denote by $A[i{\,:\,}i',j{\,:\,}j']$ the submatrix of $A$ formed by taking rows $i$ through $i'$ and columns $j$ through $j'$ (inclusive). When $i$ and $i'$ are the indices of the first and last rows, we abbreviate by $A[:,j:j']$. We abbreviate similarly for the column indices. The \emph{vectorization} of the matrix $A$ by columns, is denoted by $\vect A$. For multidimensional arrays, the notation extends in the obvious way.

Let $S$ denote the partially filled $9\times 9$ integer matrix representing the incomplete Sudoku. For convenience, let $I=\{1,2,\dots,9\}$ and let $J\subseteq I^2$ be the set of indices for which $S$ is filled.

Whilst we will formulate the problem for $9\times 9$ Sudoku, we note that the same principles can be applied to larger Sudoku puzzles.

\subsection{Sudoku modelled as integer program}\label{ssec:int}
Sudoku is modelled as an integer feasibility problem in the obvious way. Denote by $\C$, the set of vectors which are permutations of $1,2,\dots,9$. Let $A\in\R^{9\times 9}$. Then $A$ is a completion of $S$ if and only if
 $$A\in C_1\cap C_2\cap C_3\cap C_4,$$
where
 \begin{align*}
  C_1 & = \{A:A[i,:]\in \C\text{ for each }i\in I\},\\
  C_2 & = \{A:A[:,j]\in \C\text{ for each }j\in I\},\\
  C_3 & = \{A:\vect A[3i+1:3(i+1),3j+1:3(j+1)]\in \C\text{ for }i,j=0,1,2\},\\
  C_4 & = \{A:A[i,j]=S[i,j]\text{ for each }(i,j)\in J\}.
 \end{align*}

The projections onto $C_1,C_2,C_3$ are given by Proposition~\ref{prop:permutationprojections}, and can be efficiently computed by using the algorithm outlined in Remark~\ref{remark:algorithm}. The projection onto $C_4$ is given, pointwise, by
  $$(P_{C_4}A)[i,j] = \left\{\begin{array}{ll}
                       S[i,j] &\text{if }(i,j)\in J,\\
                       A[i,j] &\text{otherwise;}\\
                      \end{array}\right.$$
for each $(i,j)\in I^2$.

\subsection{Sudoku modelled as a zero-one program}\label{ssec:bin}
Denote by $\C$, the set of all $n$-dimensional standard basis vectors. To model Sudoku as a binary feasibility problem, we define $B\in\R^{9\times 9\times 9}$ by
 $$B[i,j,k] = \left\{\begin{array}{ll}
               1 & \text{if }A[i,j]=k,\\
               0 & \text{otherwise}.
              \end{array}\right.$$

Let $S'$ denote the partially filled $9\times 9\times 9$ zero-one array representing the incomplete Sudoku, $S$, under the reformulation, and let $J'\subseteq I^3$ be the set of indices for which $S'$ is filled.

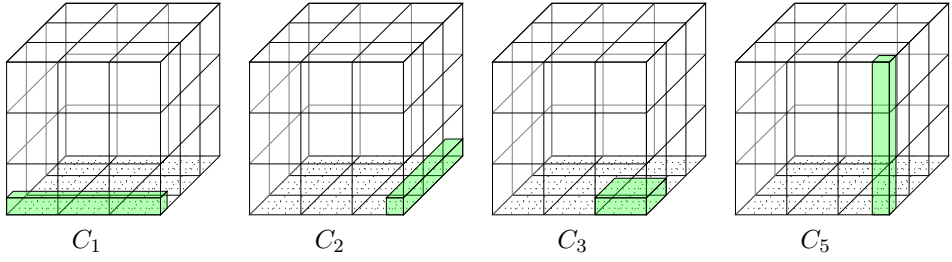
\begin{figure*}
 \begin{center}
  \scalebox{0.9}{\input{sudoku-tikz.tex}}
  \caption{Visualization of $B$ showing  constraints used in Sudoku modelled as a zero-one program. Green ``blocks" are all ``$0$", save for a single ``$1$".} \label{fig:sudoku-constraints}
 \end{center}
\end{figure*}

The four constraints of the previous section become
 \begin{align*}
  C_1 &= \{B:B[i,:,k]\in \C\text{ for each }i,k\in I\},\\
  C_2 &= \{B:B[:,j,k]\in \C\text{ for each }j,k\in I\},\\
  C_3 &= \{B:\vect B[3i+1:3(i+1),3j+1:3(j+1),k]\in \C\\\ &\quad \text{ for }i,j=0,1,2\text{ and }k\in I\},\\
  C_4 &= \{B:B[i,j,k]=1\text{ for each }(i,j,k)\in J'\}.\\
  \intertext{In addition, since each Sudoku square has precisely one entry, we require}
  C_5 &= \{B:B[i,j,:]\in \C\text{ for each }i,j\in I\}.
 \end{align*}
A visualization of the constraints is provided in Figure~\ref{fig:sudoku-constraints}.

Clearly there is a one-to-one correspondence between completed integer Sudokus, and zero-one arrays contained in the intersection of the five constraint sets. Moreover, $B$ is a completion of $S'$ if and only if
 $$B\in C_1\cap C_2\cap C_3\cap C_4\cap C_5.$$

The projections onto $C_1,C_2,C_3,C_5$ are given in Remark~\ref{remark:01projections}. The projection onto $C_4$ is given, pointwise, by
  $$(P_{C_4}B)[i,j,k] = \left\{\begin{array}{ll}
                         S[i,j,k] &\text{if }(i,j,k)\in J',\\
                         B[i,j,k] &\text{otherwise;}\\
                        \end{array}\right.$$
for each $(i,j,k)\in I^3$.

\subsection{Numerical experiments}\label{ssec:successformulation}

We have tested various large suites of Sudoku puzzles on the method of Section \ref{ssec:bin}.    We give some details regarding our implementation in C++.

\begin{itemize}
 \item \emph{Initialize:} Set $\mathbf x_0 := (y,y,y,y,y)\in D$ for some random $y\in[0,1]^{9\times 9\times 9}$.
 \item \emph{Iteration:} Set $\mathbf x_{n+1}:=T_{D,C}\mathbf x_n$.
 \item \emph{Terminate:} Either, if a solution is found, or if $10000$ iterations have been performed. More precisely, if $\operatorname{round}(P_D\mathbf x_n)$ denotes $P_D\mathbf x_n$ pointwise rounded to the nearest integer, then $\operatorname{round}(P_D\mathbf x_n)$ is a solution if
                          \begin{equation}
                           \operatorname{round}(P_D\mathbf x_n)\in C\cap D. \label{eq:terminate}
                          \end{equation}
\end{itemize}

\begin{remark}
In our implementation condition (\ref{eq:terminate}) was used a termination criterion, instead of the condition
 $$P_D\mathbf x_n \in C\cap D.$$
This improvement is due the following observation: If $P_D\mathbf x_n$ is a solution then all entries are either $0$ or $1$. \qede
\end{remark}

Since the Douglas--Rachford method produces a point whose projection onto $D$ is a solution, we also consider a variant which sets
$$
\mathbf x_{n+1}:=\left\{\begin{array}{ll}P_DT_{D,C}\mathbf x_n,
&\text{if } n\in\{400,800,1600,3200,6400\};\\
T_{D,C}\mathbf x_n, &\text{otherwise.}\end{array}\right.
$$
We will refer to this variant as \emph{DR+Proj}.

\subsubsection{Test library experience}\label{sssec:lib}
We considered Sudokus from the following libraries:
\begin{itemize}
 \item Dukuso's \texttt{top95}\footnote{\texttt{top95}: \url{http://magictour.free.fr/top95}} and \texttt{top1465}\footnote{\texttt{top1465}: \url{http://magictour.free.fr/top1465}} -- collections containing 95 and 1465 test problems, respectively. They are frequently used by programmers to test their solvers. All instances are $9\times 9$.
 \item Gordon Royle's minimum Sudoku\footnote{Gordon Royle: \url{http://school.maths.uwa.edu.au/~gordon/sudokumin.php}} -- a collection containing around 50000 distinct Sudokus with 17 entries (the best known lower bound on the number of entries required for a unique solution). All instances are $9\times 9$. Our experiments were performed on the first 1000 problems. From herein we refer to these instances as \texttt{minimal1000}.
 \item \texttt{reglib-1.3}\footnote{\texttt{reglib-1.3}: \url{http://hodoku.sourceforge.net/en/libs.php}} -- a collection containing around 1000 test problems, each suited to a particular human-style solving technique. All instances are $9\times 9$.
 \item \texttt{ksudoku16} and \texttt{ksudoku25}\footnote{\texttt{ksudoku16/25}: \url{http://carma.newcastle.edu.au/DRmethods/comb-opt/}} -- collections containing around 30 Sudokus, of various difficulties, which we generated using \emph{KSudoku}.\footnote{KSudoku: \url{http://games.kde.org/game.php?game=ksudoku}} The collections contain $16\times 16$ and $25\times 25$ instances, respectively.
 \end{itemize}

\subsubsection{Methods used for comparison}\label{sssec:alg}
Our \emph{naive} binary implementation was compared with various specialized or optimized codes. A brief description of the  methods tested follows.
 \begin{enumerate}
  \item \emph{Douglas--Rachford} in C++ -- Our implementation is outlined in Section~\ref{ssec:successformulation}. Our experiments were performed using both the normal Douglas--Rachford method (DR) and our variant (DR+Proj).
  \item \emph{Gurobi Binary Program}\footnote{Gurobi Sudoku model: \url{http://www.gurobi.com/documentation/5.5/example-tour/node155}} -- Solves a binary integer program formulation using \emph{Gurobi Optimizer 5.5}. The formulation is the same $n\times n\times n$ binary array model used in the Douglas--Rachford implementation. Our experiments were performed using the default settings, and the default settings with the pre-solver off.
  \item \emph{YASS}\footnote{YASS: \url{http://yasudokusolver.sourceforge.net/}} (Yet Another Sudoku Solver)  in C++ -- Solves the Sudoku problem in two phases. In the first phase, a \emph{reasoning algorithm} determines the possible candidates for each of the empty Sudoku squares. If the Sudoku is not completely solved, the second phase uses a deterministic recursive algorithm.
  \item \emph{DLX}\footnote{DLX: \url{http://cgi.cse.unsw.edu.au/~xche635/dlx_sodoku/}} in C -- Solves an exact cover formulation using the \emph{Dancing Links} implementation of Knuth's \emph{Algorithm X} -- a  non-deterministic, depth-first, backtracking algorithm.
 \end{enumerate}

Since YASS and DLX were only designed to be applied to $9\times 9$ instances, their performances on \texttt{ksudoku16} and \texttt{ksudoku25} were unable to be included in the comparison.

\subsubsection{Computational Results}
Table~\ref{tab:testlib1} shows a comparison of the time taken by each of the methods in Section~\ref{sssec:alg}, applied to the test libraries of Section~\ref{sssec:lib}. Computations were performed on an Intel Core i5-3210 {@} 2.50GHz running 64-bit Ubuntu~12.10. For each Sudoku puzzle, $10$ replications were performed. We make some general comments about the results.

\begin{itemize}
 \item All methods easily solved instances from \texttt{reglib-1.3} -- the test library consisting of puzzles suited to human-style techniques. Since human-style technique usually avoid excessive use of `trial-and-error', less backtracking is required to solve puzzle aimed at human players. Since all of the algorithms, except the Douglas--Rachford method, utilize some form of backtracking, this may explain the observed good performance.
 \item The Gurobi binary program performed best amongst the methods, regardless of the test library. Of the methods tested, the Gurobi Optimizer is the most sophisticated. Whether or not the pre-solver was used did not significantly effect computational time.
 \item Our Douglas--Rachford implementation outperformed YASS on \texttt{top95}, \texttt{top1465} and DLX on \texttt{minimal1000}. For all other algorithm/test library combinations, the Douglas--Rachford was competitive. The performance of the normal Douglas--Rachford method appears slightly better than the variant which includes the additional projection step.
 \item The Douglas--Rachford solved Sudoku puzzles with a high success rate -- no lower than $84\%$ for any of the test libraries. For most test libraries the success rate was much higher (see Table~\ref{tab:testlib2}). Puzzles solved by the method were typically done so in the first $2000$ iterations (see Figure~\ref{fig:sudokudist}).
\end{itemize}

 \begin{sidewaystable*}
 \caption{Mean (Max) time in seconds over all instances.} \label{tab:testlib1}
 \begin{center}
 \begin{tabular}{lcccccc} \hline\noalign{\smallskip}
                         & top95          & top1465        & reglib-1.3    & minimal1000    & ksudoku16      & ksudoku25  \\ \hline\noalign{\smallskip}
   DR                    & 1.432 (6.056)  & 0.929 (6.038)  & 0.279 (5.925) & 0.509 (5.934)  & 5.064 (30.079) & 4.011 (24.627) \\
   DR+Proj               & 1.894 (6.038)  & 1.261 (12.646) & 0.363 (6.395) & 0.953 (5.901)  & 6.757 (31.949) & 8.608 (84.190) \\
   Gurobi (default)      & 0.063 (0.095)  & 0.063 (0.171)  & 0.059 (0.123) & 0.063 (0.091)  & 0.168 (0.527)  & 0.401 (0.490) \\
   Gurobi (pre-solve off) & 0.077 (0.322)  & 0.076 (0.405)  & 0.058 (0.103) & 0.064 (0.104)  & 0.635 (4.621)  & 0.414 (0.496) \\
   YASS                  & 2.256 (58.822) & 1.440 (113.195)& 0.039 (3.796) & 0.654 (61.405) & - & - \\
   DLX                   & 1.386 (38.466) & 0.310 (34.179) & 0.105 (8.500) & 3.871 (60.541) & - & - \\ \hline
 \end{tabular}
 \end{center}

 \bigskip\bigskip\bigskip

 \caption{\% of Sudoku instances successfully solved.} \label{tab:testlib2}
 \begin{center}
 \begin{tabular}{lcccccc} \hline\noalign{\smallskip}
                         & top95          & top1465        & reglib-1.3    & minimal1000    & ksudoku16      & ksudoku25  \\ \hline\noalign{\smallskip}
   DR                    & 86.53 & 93.69 & 99.35 & 99.59 & 92.00 & 100 \\
   DR+Proj               & 85.47 & 93.93 & 99.31 & 99.59 & 84.67 & 100 \\ \hline
 \end{tabular}
 \end{center}
 \end{sidewaystable*}

\begin{sidewaysfigure*}
  \begin{subfigure}{0.5\textwidth}
  \includegraphics[width=\textwidth]{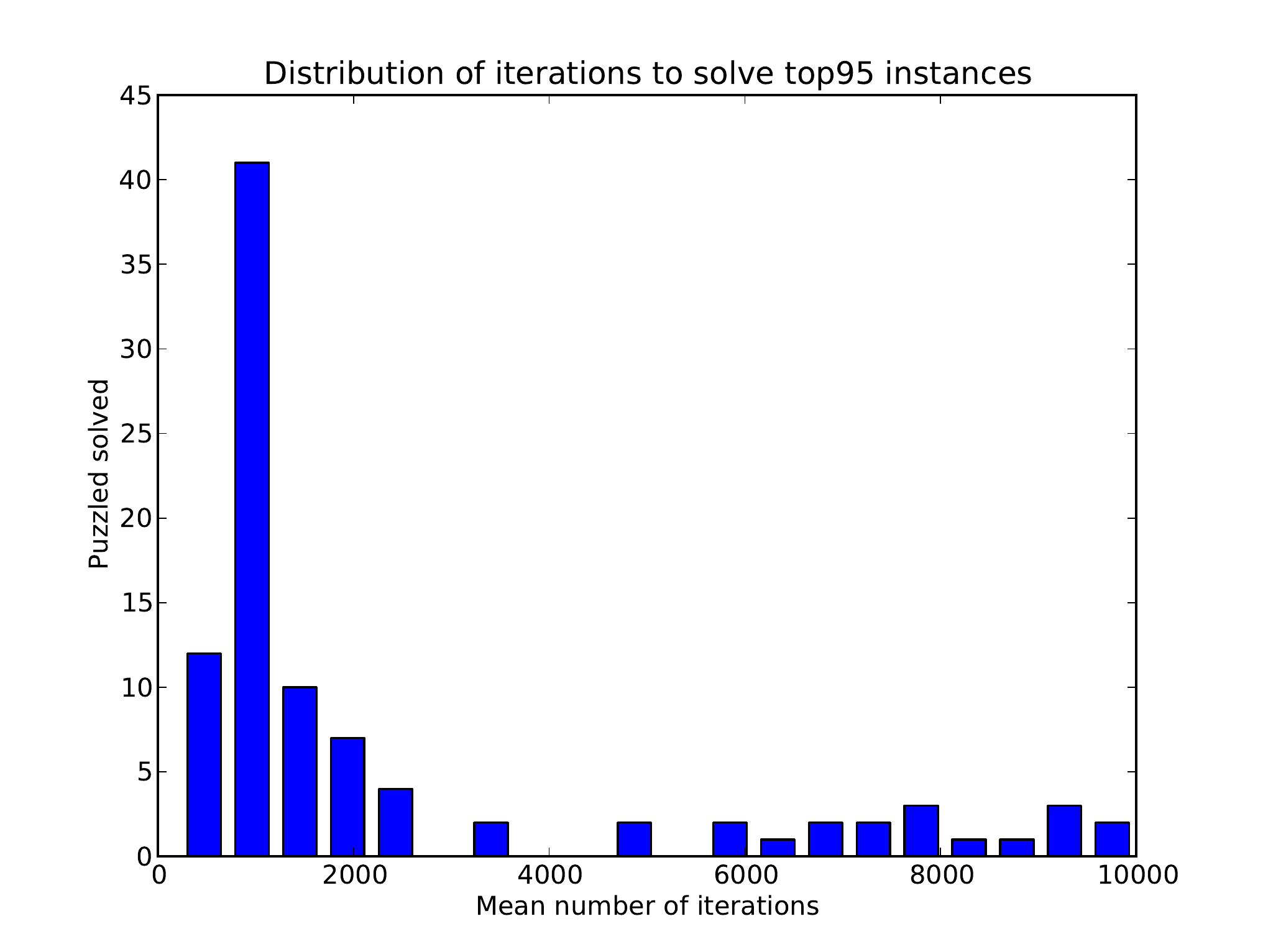}
  \end{subfigure}
  \begin{subfigure}{0.5\textwidth}
  \includegraphics[width=\textwidth]{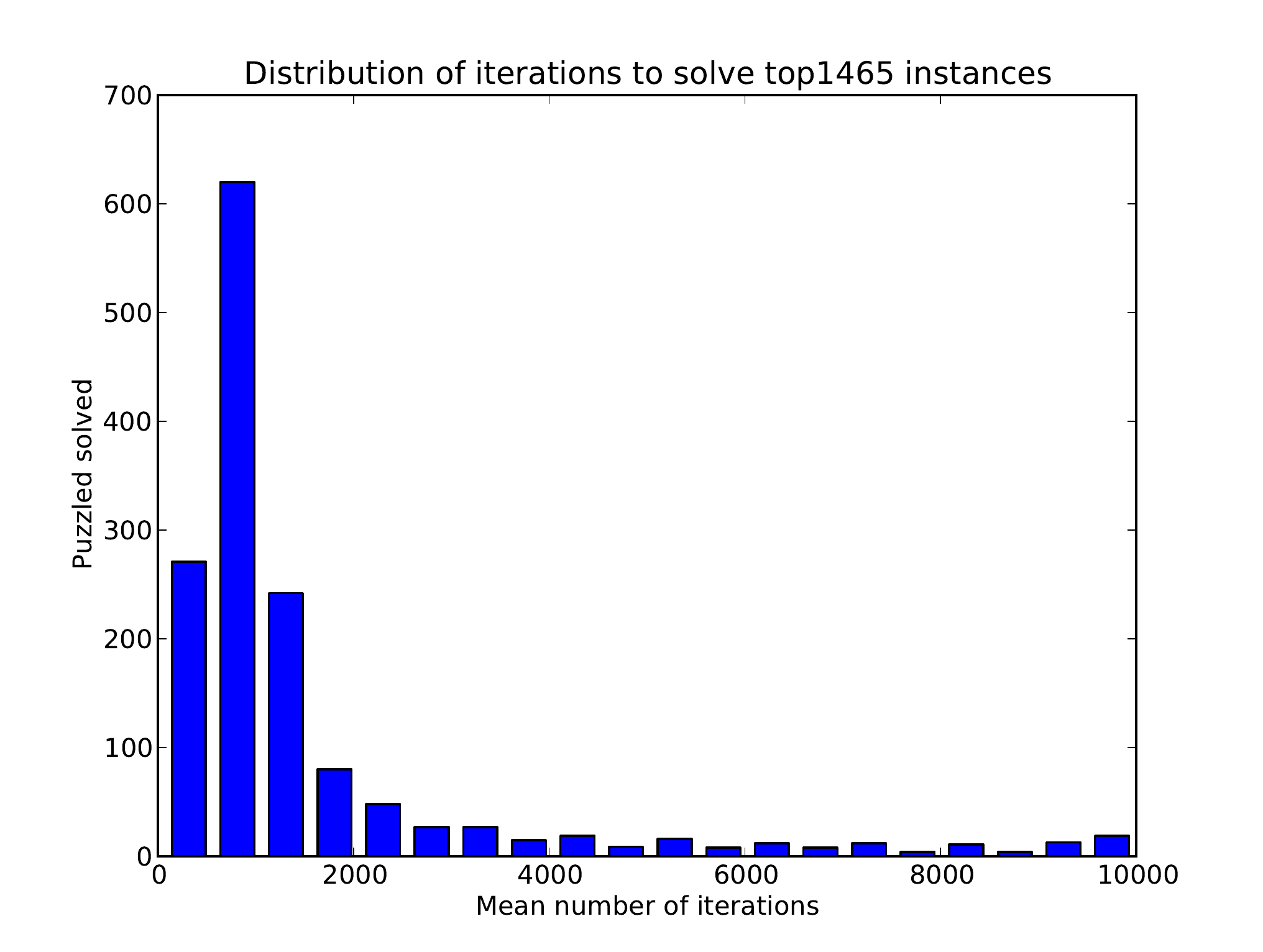}
  \end{subfigure}
  \begin{subfigure}{0.5\textwidth}
  \includegraphics[width=\textwidth]{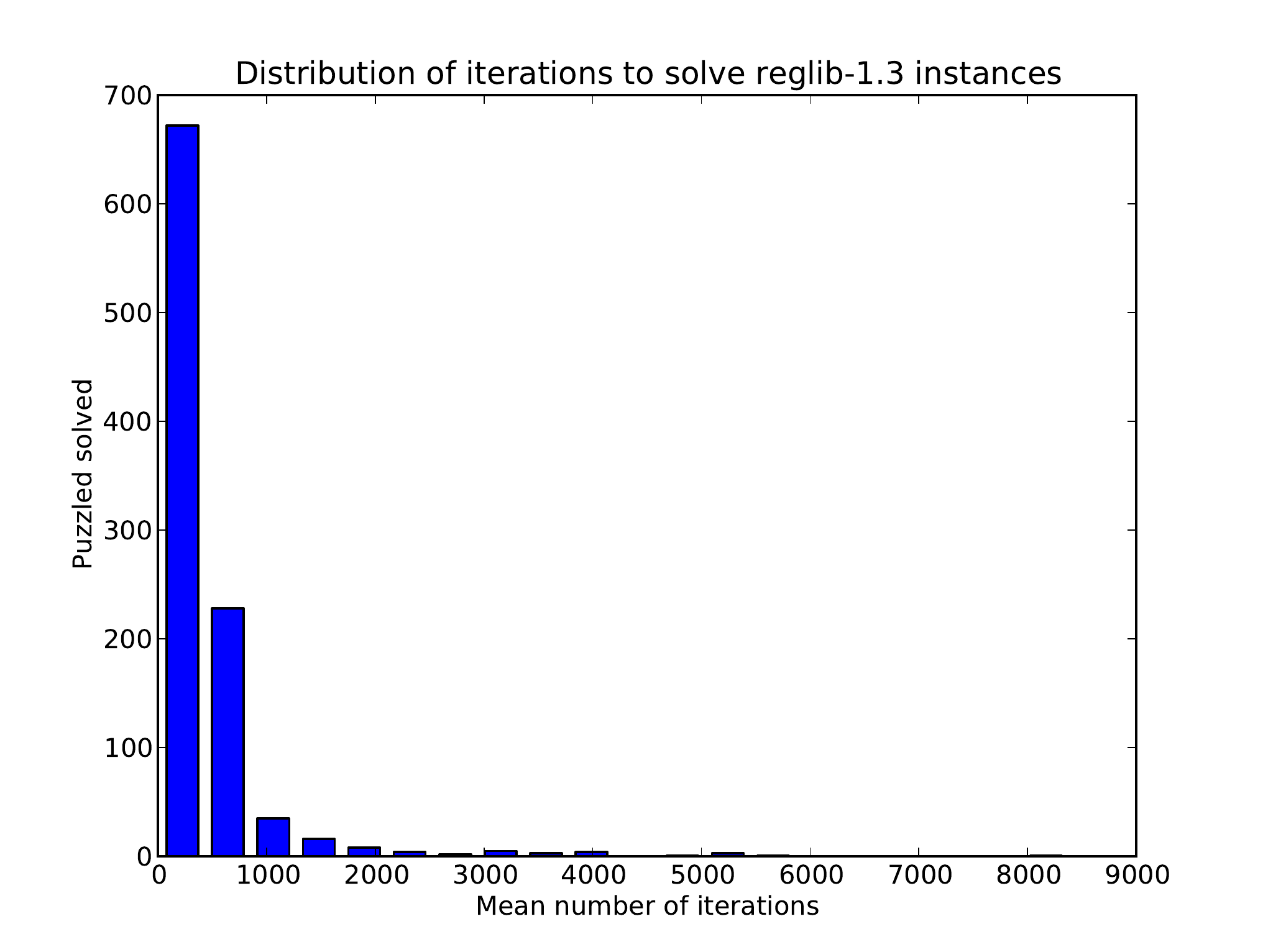}
  \end{subfigure}
  \begin{subfigure}{0.5\textwidth}
  \includegraphics[width=\textwidth]{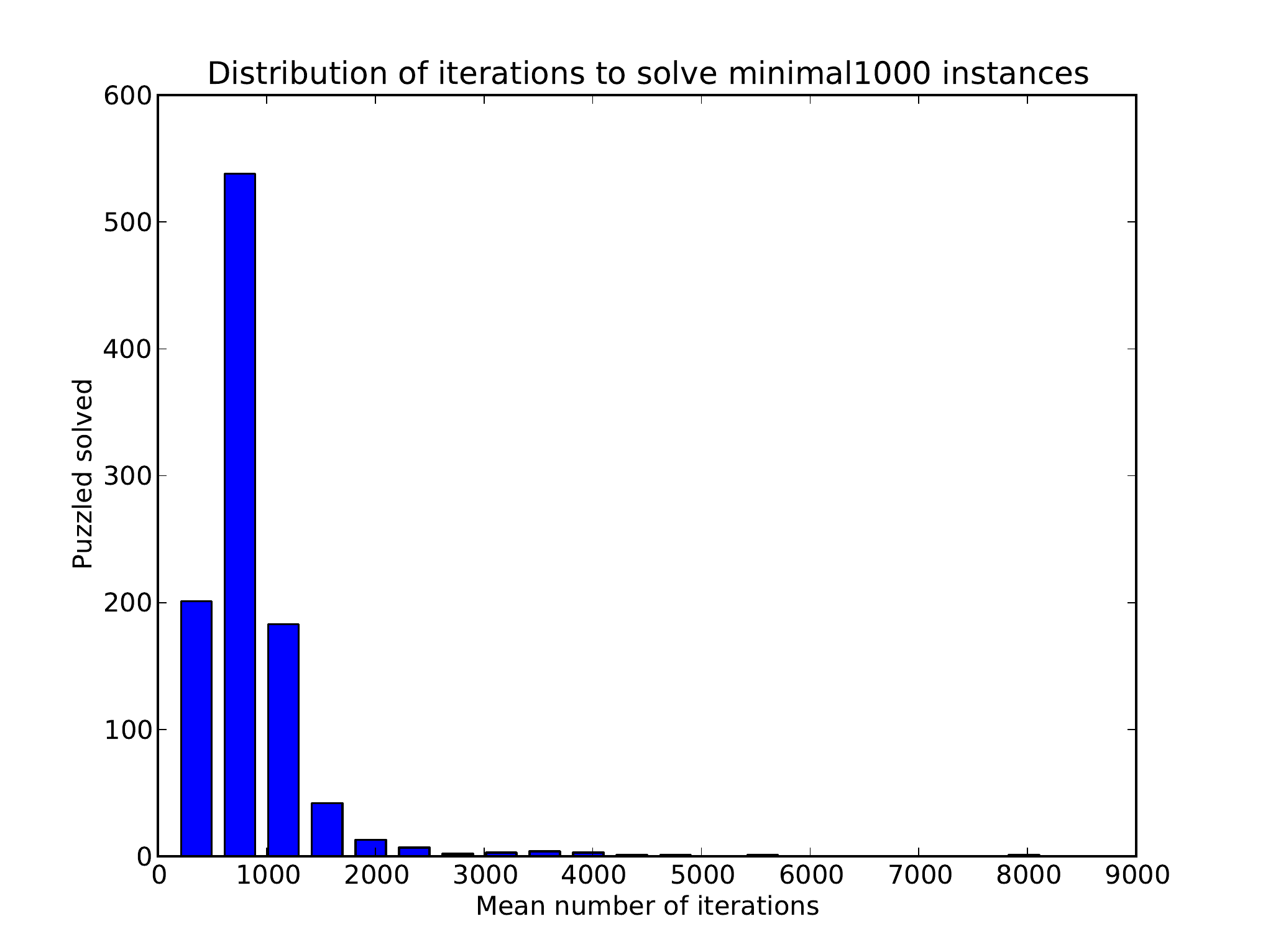}
  \end{subfigure}
  \caption{Frequency histograms showing the distribution of puzzles solved by number of iterations for the Douglas--Rachford method.} \label{fig:sudokudist}
\end{sidewaysfigure*}

\subsection{Models that failed}

To our surprise, the integer formulation of Section~\ref{ssec:int} was ineffective, except for $4\times 4$ Sudoku, while the binary reformulation of the cyclic Douglas--Rachford method described in Section~\ref{sssec:cdr} also failed in both the original space and the product space.

Clearly we have a lot of work to do to understand the model characteristics which lead  to success and those which lead to failure.

We should also like to understand how to diagnose infeasibility in Sudoku via the binary model. This would give a full treatment of Sudoku as a NP-complete problem.

\subsection{A `nasty' Sudoku puzzle and other challenges}
The incomplete Sudoku on the left of Figure \ref{fig:nasty}
has proven intractable for Douglas--Rachford. The unique solution is shown at the right of Figure \ref{fig:nasty}. As set, it can not be solved by Jason Schaad's Douglas--Rachford based Sudoku solver,\footnote{Schaad's web-based solver: \url{https://people.ok.ubc.ca/bauschke/Jason/}} nor can it be solved reliably by our implementation.


\setlength{\sudokusize}{6cm}
\renewcommand*\sudokuformat[1]{\Large\sffamily #1}
\begin{figure*}
    \begin{center}
	\scalebox{0.94}{\begin{tabular}{cc}
	\begin{minipage}{0.5\textwidth}
		\begin{sudoku}
		 |7| | | | |9| |5| |.
		 | |1| | | | | |3| |.
		 | | |2|3| | |7| | |.
		 | | |4|5| | | |7| |.
		 |8| | | | | |2| | |.
		 | | | | | |6|4| | |.
		 | |9| | |1| | | | |.
		 | |8| | |6| | | | |.
		 | | |5|4| | | | |7|.
		\end{sudoku}
	\end{minipage}
	&
	\begin{minipage}{0.5\textwidth}
		\begin{sudoku}
		 |7|{\color{red}4}|{\color{red}3}|{\color{red}8}|{\color{red}2}|9|{\color{red}1}|5|{\color{red}6}|.
		 |{\color{red}5}|1|{\color{red}8}|{\color{red}6}|{\color{red}4}|{\color{red}7}|{\color{red}9}|3|{\color{red}2}|.
		 |{\color{red}9}|{\color{red}6}|2|3|{\color{red}5}|{\color{red}1}|7|{\color{red}4}|{\color{red}8}|.
		 |{\color{red}6}|{\color{red}2}|4|5|{\color{red}9}|{\color{red}8}|{\color{red}3}|7|{\color{red}1}|.
		 |8|{\color{red}7}|{\color{red}9}|{\color{red}1}|{\color{red}3}|{\color{red}4}|2|{\color{red}6}|{\color{red}5}|.
		 |{\color{red}3}|{\color{red}5}|{\color{red}1}|{\color{red}2}|{\color{red}7}|6|4|{\color{red}8}|{\color{red}9}|.
		 |{\color{red}4}|9|{\color{red}6}|{\color{red}7}|1|{\color{red}5}|{\color{red}8}|{\color{red}2}|{\color{red}3}|.
		 |{\color{red}2}|8|{\color{red}7}|{\color{red}9}|6|{\color{red}3}|{\color{red}5}|{\color{red}1}|{\color{red}4}|.
		 |{\color{red}1}|{\color{red}3}|5|4|{\color{red}8}|{\color{red}2}|{\color{red}6}|{\color{red}9}|7|.
		\end{sudoku}
	\end{minipage}
	\\
	\end{tabular}}
	\end{center}
	\caption{The `nasty' Sudoku (left), and its unique solution (right).}\label{fig:nasty}
\end{figure*}




We decided to ask: What happens when we remove one entry from the `nasty' Sudoku? From one hundred random initializations:
\begin{itemize}
 \item Removing the top-left entry, a ``$7$", the puzzle was still difficult for the Douglas--Rachford algorithm: we had a 24\% success rate ---  comparable to the `nasty' Sudoku without any entries removed.
 \item If any other single entry was removed, the problem could be solved \emph{fairly} reliably: we had a 99\% success rate.
\end{itemize}

For each of the puzzles with an entry removed, the number of distinct solution was determined using \emph{SudokuSolver},\footnote{SudokuSolver: \url{http://infohost.nmt.edu/tcc/help/lang/python/examples/sudoku/}}
and are reported in Table~\ref{tab:numsoln}. Those with an entry removed, that could be reliably solved all have many solutions --- anywhere from a few hundred to a few thousand; while the puzzle with the top-left entry removed has relatively few --- only five.\footnote{For the five solutions: \url{http://carma.newcastle.edu.au/DRmethods/comb-opt/nasty_nonunique.txt}} It is possible that this structure that makes the `nasty' Sudoku difficult to solve, with the Douglas--Rachford algorithm hindered by an abundance of `near' solutions.

We then asked: What happens when entries from the solution are added to incomplete `nasty' Sudoku? From one hundred random starts:
\begin{itemize}
 \item If any single entry was added, the Sudoku could be solved more often, but not reliably: we had only a 54\% success rate.
\end{itemize}

\begin{table*}
 \begin{center}
 \caption{Number of instances solved from 1000 replications.}\label{tab:hardsudoku2}
 \begin{tabular}{lcc} \hline\noalign{\smallskip}
                        & AI escargot   & `Nasty' \\ \hline\noalign{\smallskip}
  DR                    & 985           & 202 \\
  DR+Proj               & 975           & 172 \\ \hline
 \end{tabular}
 \end{center}
\end{table*}


We also examined how the binary Douglas--Rachford method applied to this `nasty' Sudoku behaves relative to its behaviour on other hard problems (see Table~\ref{tab:hardsudoku2}). Specially, we considered \emph{AI escargot}, a Sudoku purposely designed by Arto Inkala to be really difficult. Our Douglas--Rachford implementation could solve AI escargot \emph{fairly} reliably: we had a success rate of 99\%. In contrast to the `nasty' Sudoku, the number of solutions to AI escargot with one entry removed was no more than a few hundred; typically much less.

\begin{figure*}[t]
 \begin{center}
  \includegraphics[width=0.65\textwidth]{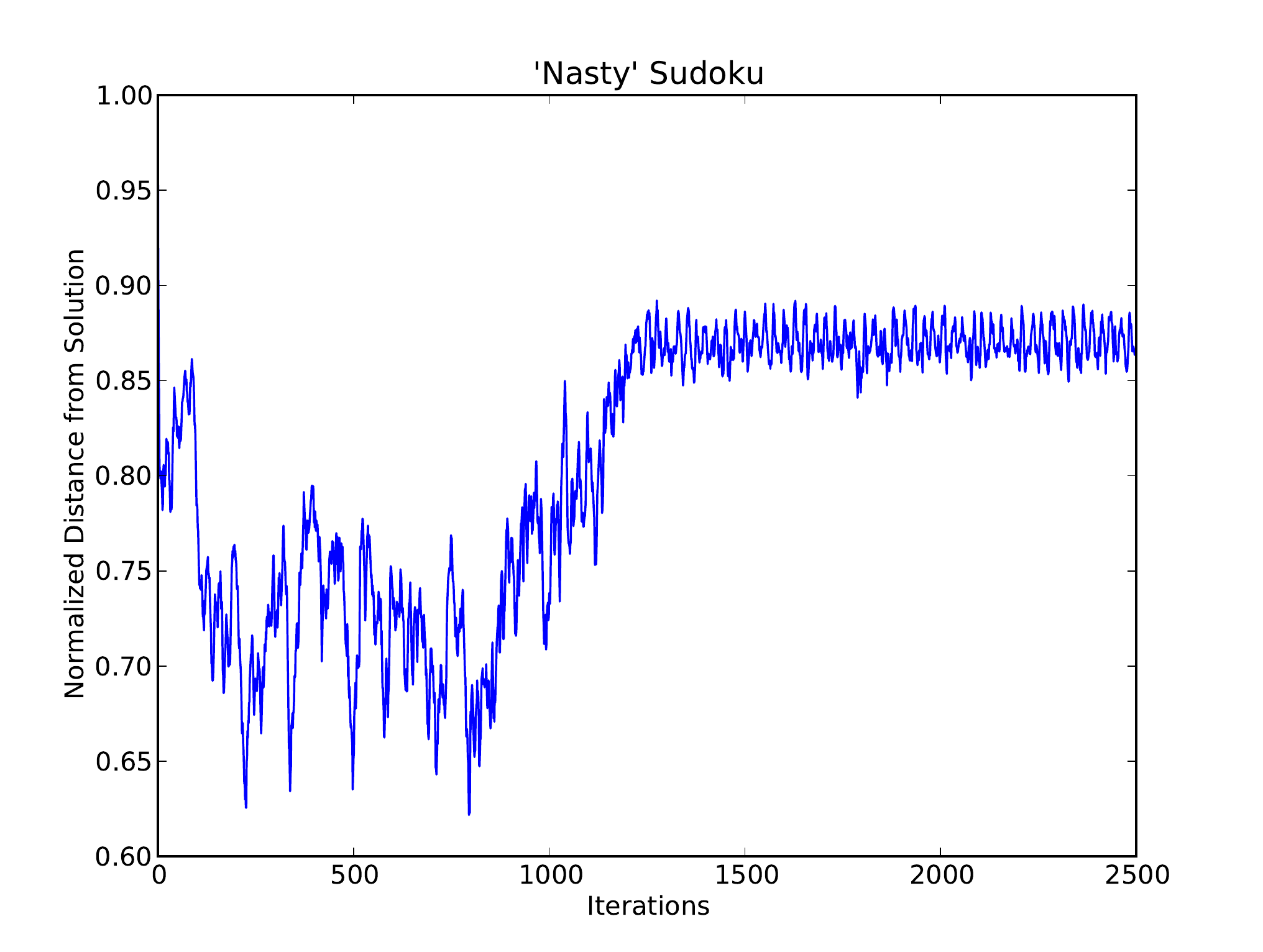}
 \end{center}
 \caption{Typical behaviour of the Douglas--Rachford algorithm applied to the `nasty' Sudoku, modelled as a zero-one program.}\label{fig:sudoku1}
\end{figure*}

\begin{figure*}
 \begin{center}
  \includegraphics[width=0.65\textwidth]{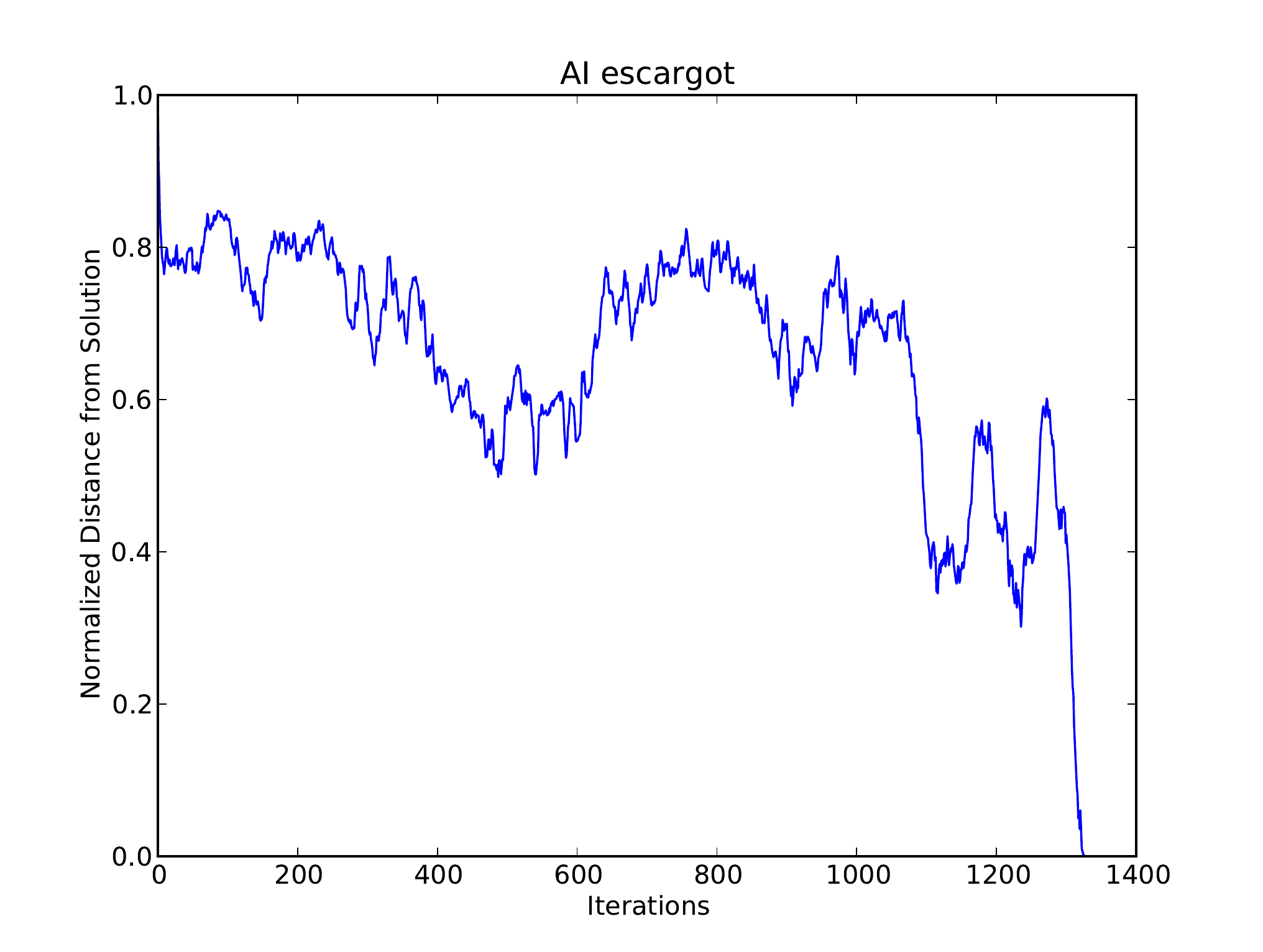}
 \end{center}
 \caption{Typical behaviour of the Douglas--Rachford algorithm for AI escargot, modelled as a zero-one program.}\label{fig:sudoku2}
\end{figure*}

We then asked the question: How does the distances from the solution vary as a function of the number of iterations? This is plotted in Figures~\ref{fig:sudoku1} and \ref{fig:sudoku2}, for the `nasty' Sudoku and AI escargot, respectively.\footnote{If $\mathbf x_n$ is the current iterate, $\mathbf x^\ast$ the solution, and $m=\max_n\|P_D \mathbf x_n-\mathbf x^\ast\|$, ${\|P_D\mathbf x_n- \mathbf x^\ast\|/m}$ is plotted against $n$.} The same for each of the five solution to the `nasty' Sudoku, with the top-left entry removed, is shown in Figure~\ref{fig:sudoku-nonunique}.

\begin{table*}
\begin{center}
\caption{Number of distinct solutions for the `nasty' Sudoku with a single entry removed.}\label{tab:numsoln}
\begin{tabular}{cc} \hline
 Entry removed & Distinct solutions \\ \hline\noalign{\smallskip}
 None          & $1$ \\
 $S[1,1]$      & $5$ \\
 $S[1,6]$      & $571$ \\
 $S[1,8]$      & $2528$ \\
 $S[2,2]$      & $874$  \\
 $S[2,8]$      & $1504$ \\
 $S[3,3]$      & $2039$ \\
 $S[3,4]$      & $1984$ \\
 $S[3,7]$      & $182$ \\
 $S[4,3]$      & $2019$ \\
 $S[4,4]$      & $3799$ \\
 $S[4,8]$      & $1263$ \\ \hline
 \end{tabular}
 \hfill
 \begin{tabular}{cc} \hline
 Entry removed &  Distinct solutions \\  \hline\noalign{\smallskip}
 $S[5,1]$      & $216$ \\
 $S[5,7]$      & $2487$ \\
 $S[6,6]$      & $476$ \\
 $S[6,7]$      & $1315$ \\
 $S[7,2]$      & $1905$ \\
 $S[7,5]$      & $966$ \\
 $S[8,2]$      & $711$ \\
 $S[8,5]$      & $579$ \\
 $S[9,3]$      & $1278$\\
 $S[9,4]$      & $1368$ \\
 $S[9,9]$      & $1640$ \\
 & \\\hline
\end{tabular}
\end{center}
\end{table*}

In what follows, denote by $(\mathbf x_n)$ the sequence of iterates obtained from the Douglas--Rachford algorithm, and by $\mathbf x^\ast$ the Sudoku solution obtained from ($\mathbf x_n$). In contrast to the convex setting, Figures~\ref{fig:sudoku1} and \ref{fig:sudoku2} show that the sequence $(\|\mathbf x_n-\mathbf x^\ast\|)$ need not be monotone decreasing.

In the convex setting, $(\mathbf x_n)$ is known to have the very useful property of being \emph{Fej\'er monotone} with respect to $\Fix T_{D,C}$. That is,
 $$\|\mathbf x_{n+1}-c\|\leq \|\mathbf x_n-c\|\text{ for any }c\in\Fix T_{D,C}.$$
When $(\mathbf x_n)$ converged to a solution, $\|\mathbf x_n-x^\ast\|$ decreased rapidly just before the solution was found (see Figure~\ref{fig:sudoku2}). This seemed to occur regardless of the behaviour of earlier iterations. Perhaps this behaviour is due to the Douglas--Rachford iterate entering a local basin of attraction.




 \begin{figure*}
  \begin{center}
   \includegraphics[width=0.75\textwidth]{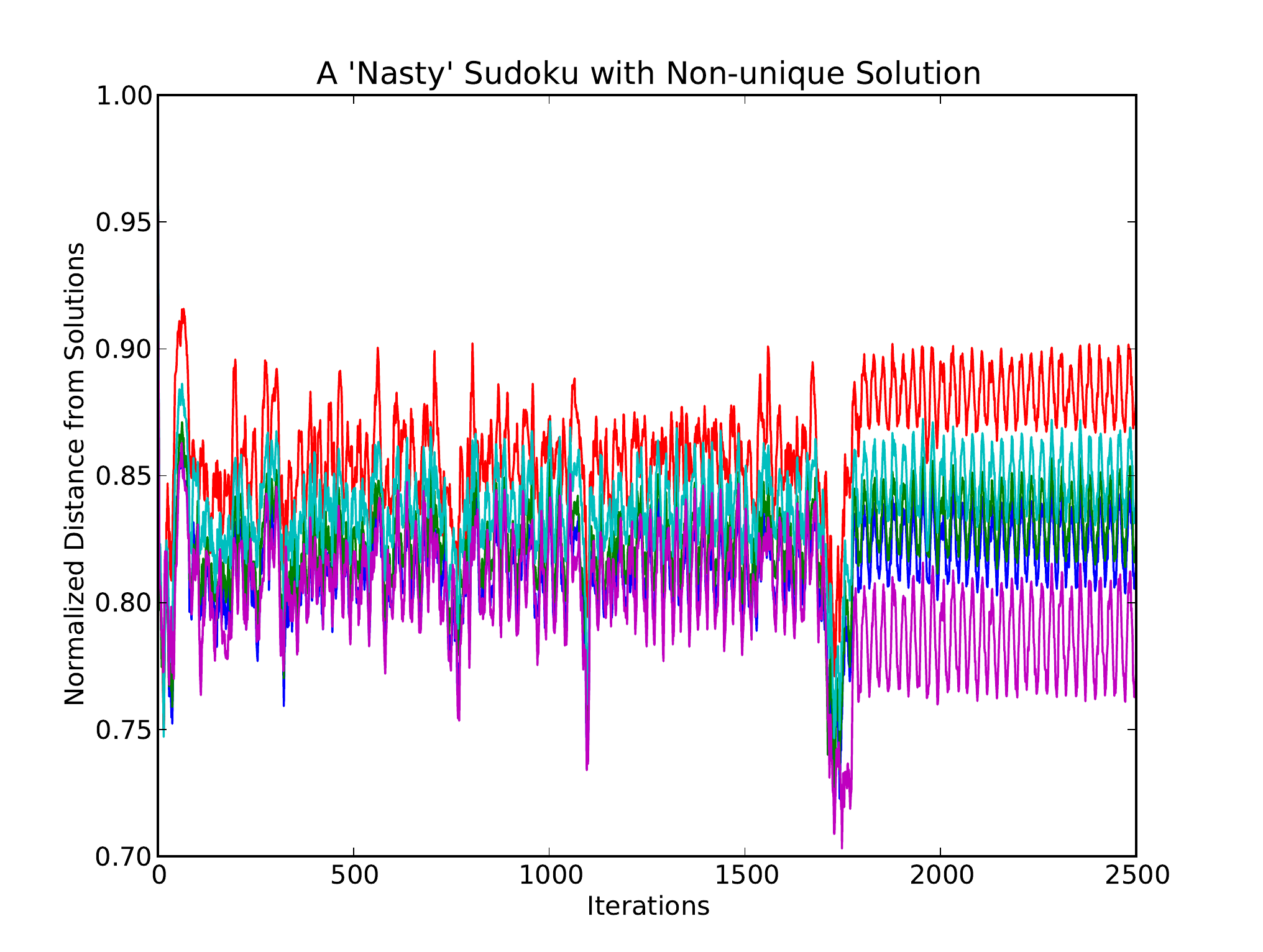}
   \caption{Typical behaviour of Douglas--Rachford  applied to  `nasty' Sudoku with  top-left entry removed. The five colors represent the possible solutions.}\label{fig:sudoku-nonunique}
  \end{center}
 \end{figure*}
 
\begin{table*}
 \begin{center}
 \caption{Mean (Max) Time in second from 1000 replications.} \label{tab:hardsudoku1}
 \begin{tabular}{lcc} \hline\noalign{\smallskip}
                        & AI escargot   & `Nasty' \\ \hline\noalign{\smallskip}
  DR                    & 1.232 (6.243) &  4.840 (6.629) \\
  DR+Proj               & 1.623 (6.074) &  5.312 (7.689) \\
  Gurobi (default)      & 0.157 (0.845) & 0.111 (0.125) \\
  Gurobi (pre-solve off) & 0.094 (0.153) & 0.253 (0.365) \\
  YASS                  & 0.162 (0.255) & 12.370 (13.612)\\
  DLX                   & 0.020 (0.032) & 0.110 (0.126) \\ \hline
 \end{tabular}
 \end{center}
\end{table*}

The methods Section~\ref{sssec:alg}, applied to the two difficult Sudoku puzzles, were also compared (see Table~\ref{tab:hardsudoku1}). While all solved AI escargot easily, applied to the `nasty' Sudoku, YASS was significantly slower -- the Douglas--Rachford method is not the only algorithm to find the puzzle difficult.



\section{Solving Nonograms}\label{sec:vis}
Recall that a nonogram puzzle consists of a blank $m\times n$ grid of pixels (the canvas) together with $(m+n)$ cluster-size sequences, one for each row and each column \cite{bosch2001painting}. The goal is to paint the canvas with a picture that satisfies the following constraints:
\begin{itemize}
 \item Each pixel must be black or white.
 \item If a row (resp. column) has cluster-size sequence $s_1,s_2,\dots,s_k$ then it must contain $k$ clusters of black pixels, separated by at least one white pixel, such that the $i$th leftmost (resp. uppermost) cluster contains $s_i$ black pixels.
\end{itemize}
An example of a nonogram puzzle is given in Figure~\ref{fig:happy}. Its solution, found by the Douglas--Rachford algorithm, is shown in Figure~\ref{fig:happy_sol}.

\begin{figure*}
  \begin{center}
     \scalebox{0.9}{\input{camel_0.tex}}
  \end{center}
  \caption{A nonogram  whose solution can be found by Douglas--Rachford, see Figure~\ref{fig:happy_sol}.  Cluster-size sequences for each row and column are given.} \label{fig:happy}
\end{figure*}

\begin{figure*}[t]
  \begin{center}
	  \begin{subfigure}{0.3\textwidth}
	    \includegraphics[width=\textwidth]{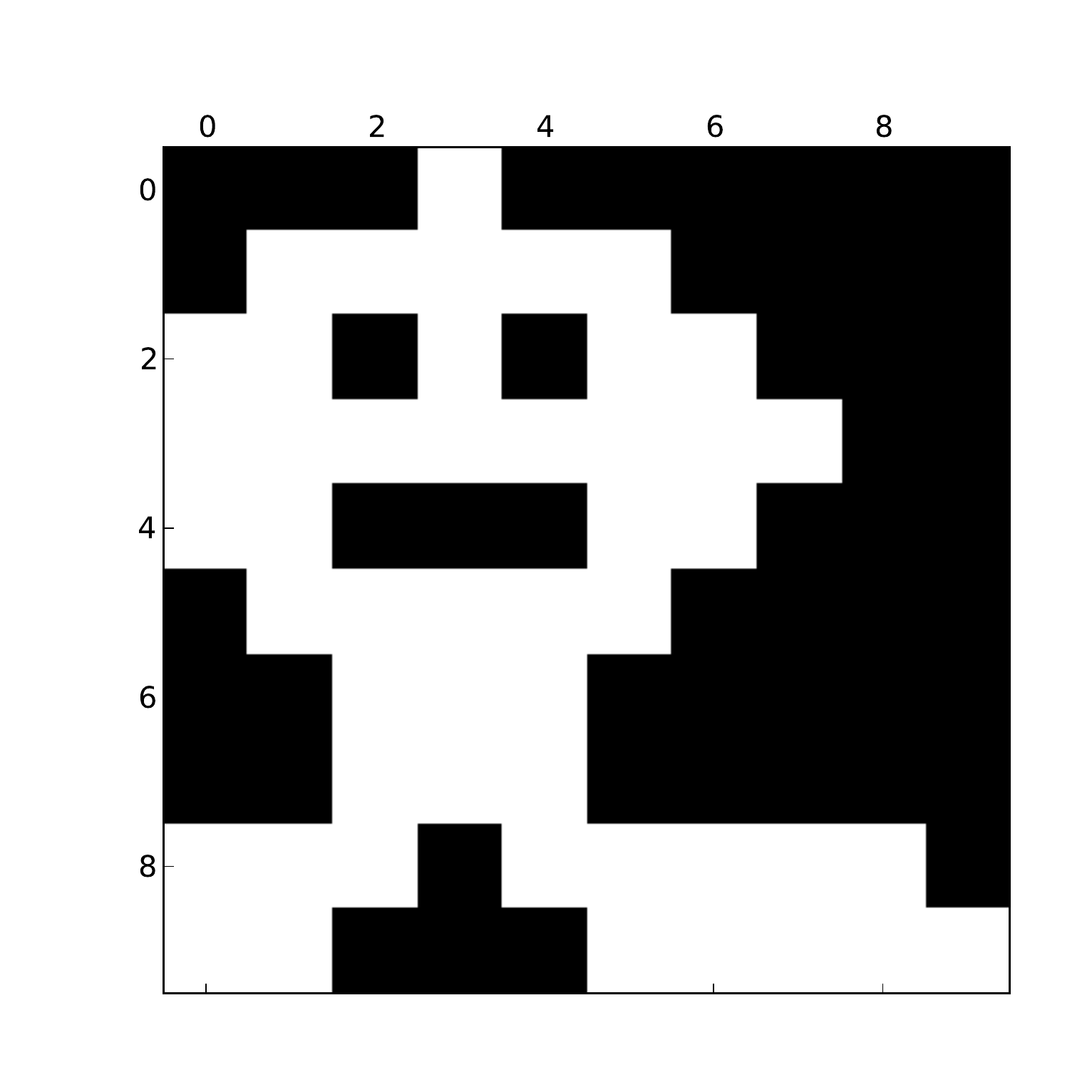}
	    \caption{A spaceman.}
	  \end{subfigure}
	  \begin{subfigure}{0.3\textwidth}
         \includegraphics[width=\textwidth]{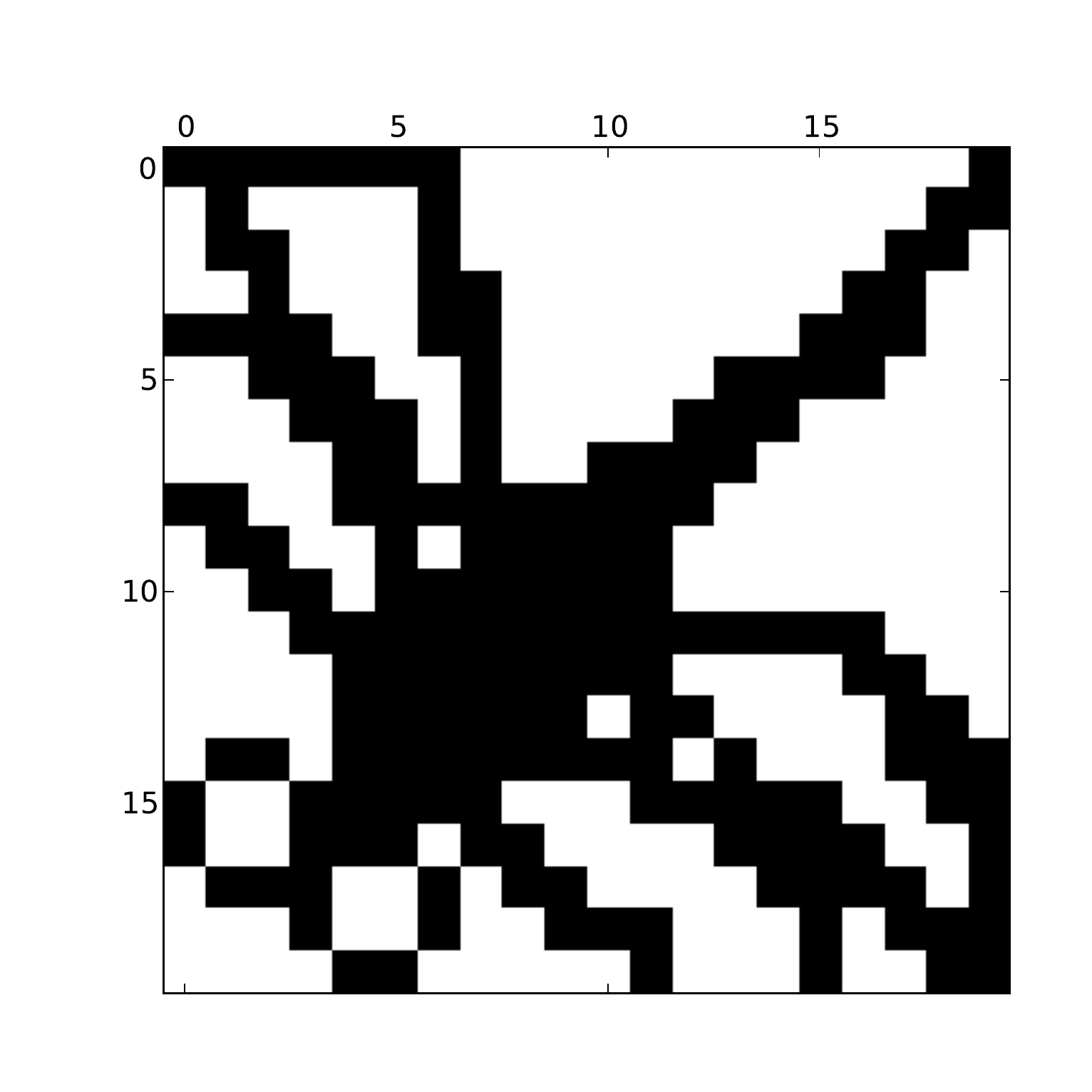}
         \caption{A dragonfly.}
	  \end{subfigure}
	  \begin{subfigure}{0.3\textwidth}
         \includegraphics[width=\textwidth]{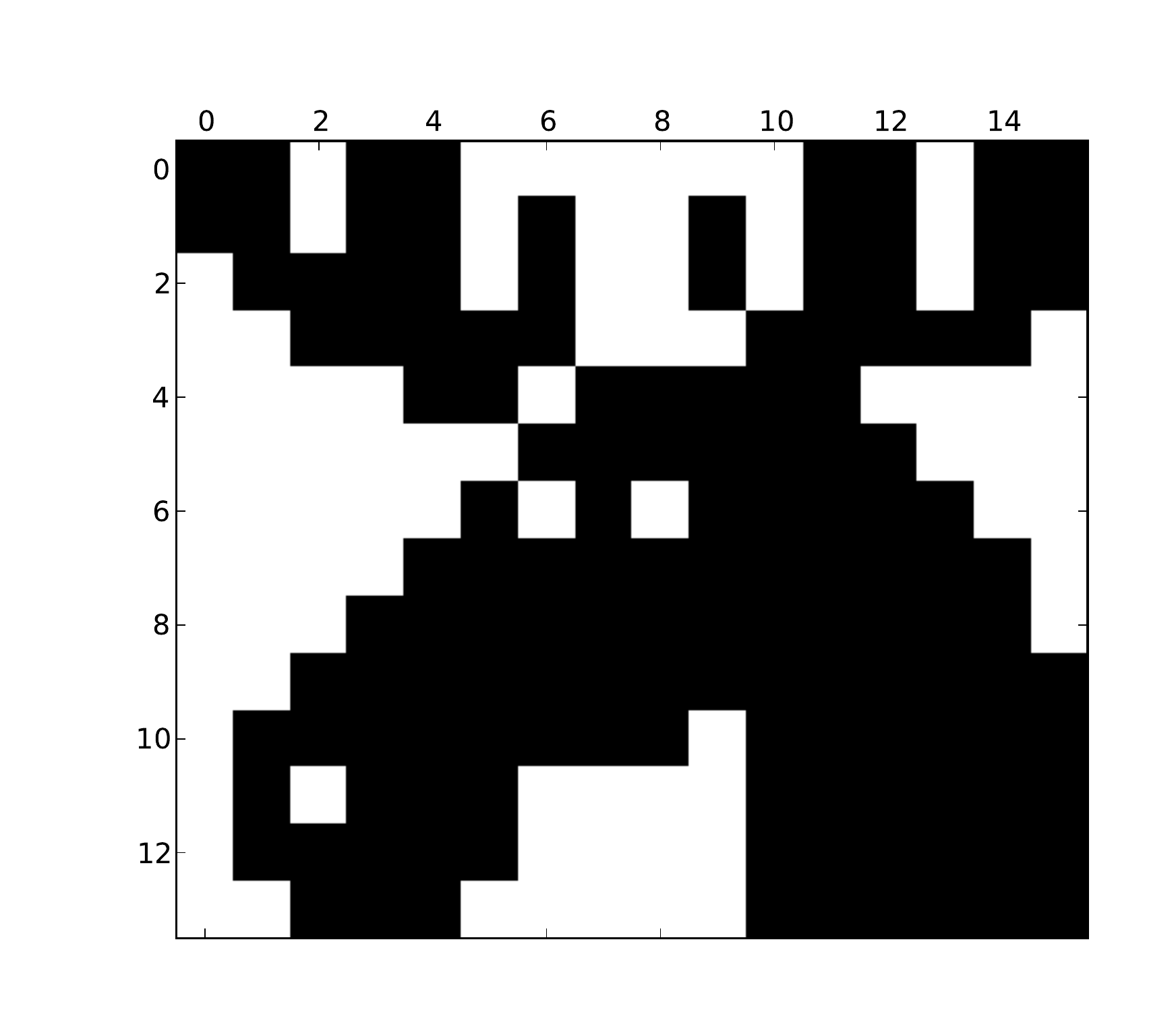}
         \caption{A moose.}
	  \end{subfigure}
	  \begin{subfigure}{0.3\textwidth}
         \includegraphics[width=\textwidth]{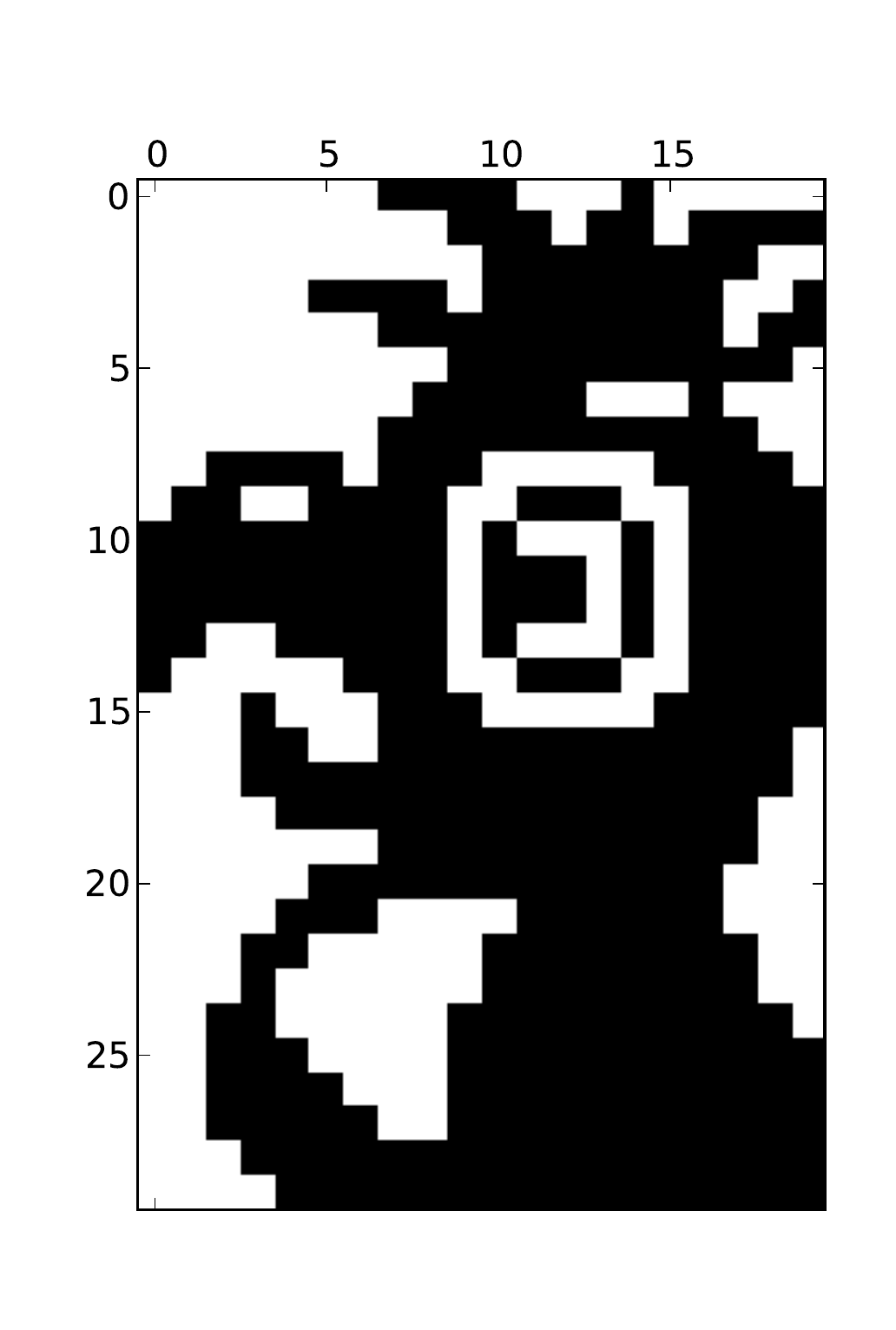}
         \caption{A parrot.}
	  \end{subfigure}
	  \begin{subfigure}{0.3\textwidth}
         \includegraphics[width=\textwidth]{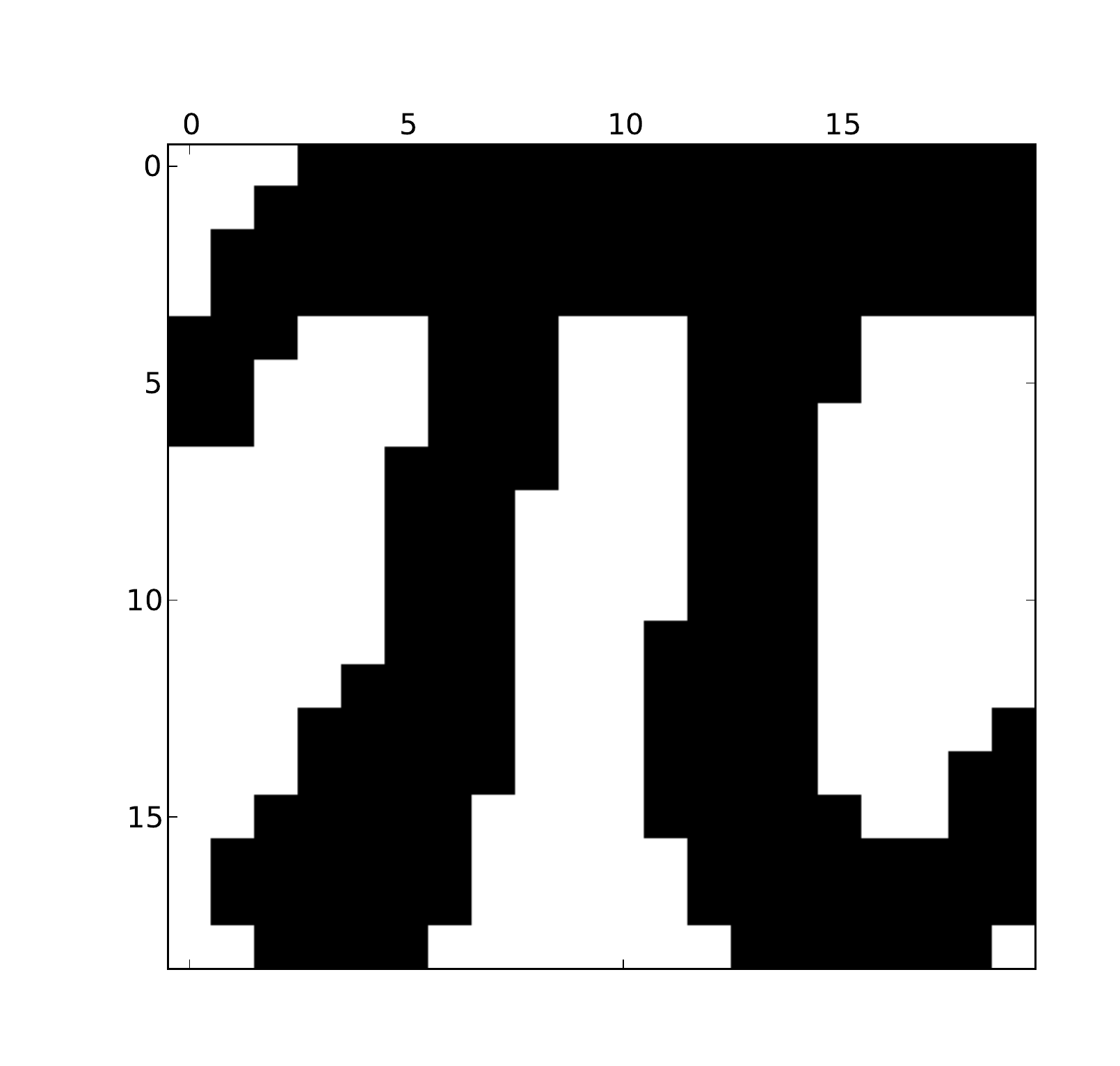}
         \caption{The number $\pi$.}
	  \end{subfigure}
	  \begin{subfigure}{0.3\textwidth}
         \includegraphics[width=\textwidth]{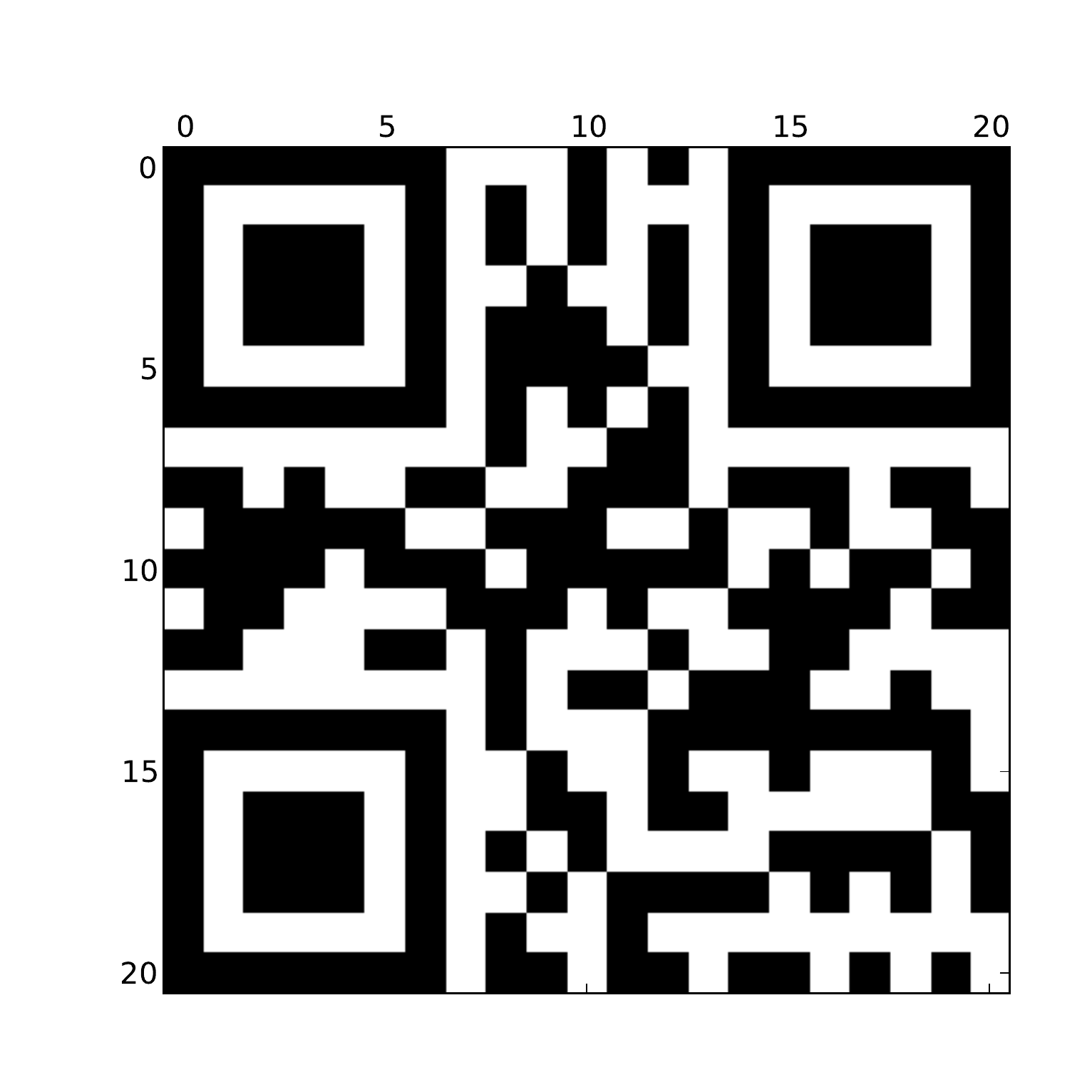}
         \caption{``Hello from CARMA" encoded as a QR code.\footnotemark}
	  \end{subfigure}	  	
  \end{center}
  \caption{Solutions to six nonograms found by the Douglas--Rachford algorithm.}\label{fig:6nonograms}
\end{figure*}

We model nonograms as a binary feasibility problem. The $m\times n$ grid is represented as a matrix $A\in\R^{m\times n}$. We define
 $$A[i,j] = \left\{\begin{array}{ll}
             0 & \text{if the }(i,j)\text{-th entry of the grid is white}, \\
             1 & \text{if the }(i,j)\text{-th entry of the grid is black}. \\
            \end{array}\right.$$


Let $\mathcal R_i\subset\R^m$ (resp. $\mathcal C_j\subset\R^n$) denote the set of vectors having cluster-size sequences matching row $i$ (resp. column $j$).
\begin{align*}
 C_1 &= \{A: A[i,:]\in\mathcal R_i\text{ for }i=1,\dots,m\}, \\
 C_2 &= \{A: A[:,j]\in\mathcal C_j\text{ for }j=1,\dots,n\}.
\end{align*}
Given an incomplete nonogram puzzle, $A$ is a solution if and only if
 $$A\in C_1\cap C_2.$$

We investigated the viability of the Douglas--Rachford method to solve nonogram puzzles, by testing the algorithm on seven puzzles: the puzzle in Figure~\ref{fig:happy}, and the six puzzles shown in Figure~\ref{fig:6nonograms}. Our implementation, written in Python, is, appropriately modified, the same as the method of Section~\ref{ssec:successformulation}.

\footnotetext{\emph{QR (quick response) codes} are two-dimensional bar codes originally designed for use in the Japanese automobile industry. Their data is typically encoded in either numerical, alphanumerical, or binary formats.}

Applied to nonograms, the Douglas--Rachford algorithm is highly successful. From 1000 random initializations, all puzzles considered were solved with a 100\% success rate.

Within this model, a difficulty is that the projections onto $C_1$ and $C_2$ have no simple form. So far, our attempts to find an efficient method to do so have been unsuccessful. Our current implementation pre-computes $\mathcal R_i$ and $\mathcal C_j$, for all indices $i,j$, and at each iteration chooses the nearest point by computing the distance to each point in the appropriate set.

For nonograms with large canvases, the enumeration of $\mathcal R_i$ and $\mathcal C_j$ becomes intractable. However, the Douglas--Rachford iterations themselves are fast.

\begin{remark}[Performance on NP-complete problems]  We note that for Sudoku, the computation of projections is easy but the typical number of (easy) iterative steps large---as befits an NP complete problem. By contrast for nonograms, the number of steps is very small but an exponential amount of work is presumably buried in computing the projections. \qede\end{remark}

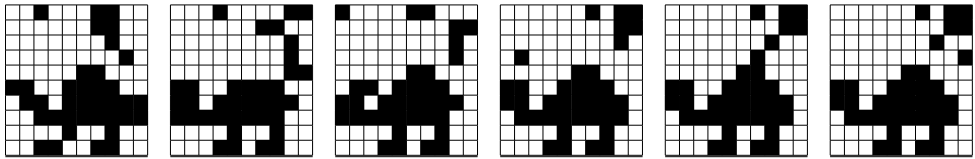
\begin{figure*}[]
  \begin{center}
    \scalebox{0.22}{\input{camel_1.tex}}\hspace{4pt}
    \scalebox{0.22}{\input{camel_2.tex}}\hspace{4pt}
    \scalebox{0.22}{\input{camel_3.tex}}\hspace{4pt}
    \scalebox{0.22}{\input{camel_4.tex}}\hspace{4pt}
    \scalebox{0.22}{\input{camel_5.tex}}\hspace{4pt}
    \scalebox{0.22}{\input{camel_6.tex}}
  \end{center}
  \caption{Solution to the nonogram in Figure~\ref{fig:happy} found by Douglas--Rachford in six iterations: showing the projection onto $C_1$ of these six iterations.} \label{fig:happy_sol}
\end{figure*}

\section{Conclusion}\label{sec:conc}

The message of the list in Section \ref{ssec:list} and of the previous two sections is the following.   When presented with a new combinatorial feasibility problem it is well worth seeing if Douglas--Rachford can deal with it---it is conceptually very simple and is usually relatively easy to implement. It would be interesting to apply Douglas--Rachford to various other classes of matrix-completion problem \cite{matrix}.

Moreover, this approach allows for the intuition developed in Euclidean space to be usefully repurposed. This lets one profitably consider non-expansive fixed point methods in the class of CAT(0) metric spaces --- a far ranging concept introduced twenty years ago in algebraic topology but now finding applications to optimization and fixed point algorithms.
The convergence of various projection type algorithms to feasible points is under investigation by Searston and Sims among others in such spaces \cite{cat}: thereby broadening the constraint structures to which projection-type algorithms apply to include metrically rather than only algebraically convex sets.

Weak convergence of project-project-average has been established  \cite{cat}. Reflections have been shown to be well defined in those CAT(0) spaces with extensible geodesics and curvature bounded below \cite{cat2}. Examples have been constructed to show that unlike in Hilbert spaces they need not be nonexpansive unless the space has constant curvature \cite{cat2}. None-the-less it appears that the basic Douglas--Rachford algorithm (reflect-reflect-average) may continue to converge in fair generality.

\bigskip

Many resources can be found at the paper's companion website:\\
\centerline{\url{http://carma.newcastle.edu.au/DRmethods/comb-opt/}}

\vfill
\paragraph{Acknowledgements} We wish to thank Heinz Bauschke, Russell Luke, Ian Searston and Brailey Sims for many useful insights. Example~\ref{ex:3setdr} was provided by Brailey Sims.

\footnotesize

 \bibliographystyle{plain}

\input{DR-ref.tex}
\end{document}

%% file: 3setdr-v2.tex
\begin{tikzpicture}[scale=3]
 \tikzset{>=stealth',every on chain/.append style={join},every join/.style={->}}

 \draw[<->,color=blue,thick] (0,1.5) -- (0,-1);
 \draw[<->,color=red,thick] (-1.299,-0.75) -- (0.866,0.5);
 \draw[<->,color=green,thick] (1.299,-0.75) -- (-0.866,0.5);
 
 \draw (0,1.25) node[right] {{\color{blue}$A$}};
 \draw (0.65,0.35) node[below right] {{\color{red}$B$}};
 \draw (-0.65,0.35) node[below left] {{\color{green}$C$}};
 
 \draw[->] (-0.866,-0.5) -- (0.866,-0.5);
 \draw[->] (0.866,-0.5) -- (0,1);
 \draw[->] (0,1) -- (-0.866,-0.5);
 
 \filldraw[black] (-0.866,-0.5) circle(0.4pt);
 \filldraw[black] (0.866,-0.5) circle(0.4pt);
 \filldraw[black] (0,1) circle(0.4pt);
 \filldraw[black] (0,0) circle(0.4pt);
 
 \draw (-0.55,-0.6) node[below] {$x_0=R_CR_BR_Ax_0$};
 \draw (0.866,-0.5) node[above right] {$R_Ax_0$};
 \draw (0,1) node[above left] {$R_BR_Ax_0$};
 \draw (0.05,0) node[right] {$(0,0)$};
\end{tikzpicture}

%% file: cycdr-ex.tex
\begin{tikzpicture}[scale=3]
 \tikzset{>=stealth',every on chain/.append style={join},every join/.style={->}}

 \draw[<->,color=blue,thick] (0,1.5) -- (0,-1);
 \draw[<->,color=red,thick] (-1.299,-0.75) -- (0.866,0.5);
 \draw[<->,color=green,thick] (1.299,-0.75) -- (-0.866,0.5);
 
 \draw (0,1.25) node[right] {{\color{blue}$A$}};
 \draw (0.65,0.35) node[below right] {{\color{red}$B$}};
 \draw (-0.65,0.35) node[below left] {{\color{green}$C$}};
 

\draw[dotted,->,black!50] (-0.87, -0.5) -- (0.87, -0.5);
\draw[dotted,->,black!50] (0.87, -0.5) -- (0.0, 1.0138);
\draw[dotted,black!50] (-0.87, -0.5) -- (0.0, 1.0138);
\draw[dotted,->,black!50] (-0.435, 0.25690000000000002) -- (0.0060030000000000361, -0.51044522000000003);
\draw[dotted,->,black!50] (0.0060030000000000361, -0.51044522000000003) -- (0.44708884139999999, 0.25704414403599984);
\draw[dotted,black!50] (-0.435, 0.25690000000000002) -- (0.44708884139999999, 0.25704414403599984);
\draw[dotted,->,black!50] (0.0060444206999999972, 0.25697207201799993) -- (-0.22054349230565995, -0.1372908966118484);
\draw[dotted,->,black!50] (-0.22054349230565995, -0.1372908966118484) -- (0.22054349230565995, -0.1372908966118484);
\draw[dotted,black!50] (0.0060444206999999972, 0.25697207201799993) -- (0.22054349230565995, -0.1372908966118484);
 
 \filldraw[black] (-0.866,-0.5) circle(0.4pt);
 \filldraw[black] (-0.435, 0.256) circle(0.4pt);
 \filldraw[black] (0, 0.257) circle(0.4pt);
 \filldraw[black] (0.113, 0.0598) circle(0.4pt);

 \draw (-0.87, -0.5) node[below right] {$x_0$};
 \draw (-0.36, 0.256) node[above] {$T_{A,B}x_0$};
 \draw (0, 0.5) node[right] {$T_{B,C}T_{A,B}x_0$};
 \draw (0.113, 0) node[right] {$T_{[A\,B\,C]}x_0=T_{C,A}T_{B,C}T_{A,B}x_0$};
 \draw (-1.8,0) node {}; 

\draw[->] (-0.87, -0.5) -- (-0.435, 0.25690000000000002);
\draw[->] (-0.435, 0.25690000000000002) -- (0.0060444206999999972, 0.25697207201799993);
\draw[->] (0.0060444206999999972, 0.25697207201799993) -- (0.11329395650282997, 0.059840587703075765);
\draw[->] (0.11329395650282997, 0.059840587703075765) -- (0.054354144776545446, -0.034735624208113311);
\draw[->] (0.054354144776545446, -0.034735624208113311) -- (-0.0019175000879933131, -0.032777392390574284);
\draw[->] (-0.0019175000879933131, -0.032777392390574284) -- (-0.014737540711898141, -0.007134071551871516);
\draw[->] (-0.014737540711898141, -0.007134071551871516) -- (-0.0067877063030386438, 0.0046765374154157258);
\draw[->] (-0.0067877063030386438, 0.0046765374154157258) -- (0.00038885804079640289, 0.0041791125505887617);
\draw[->] (0.00038885804079640289, 0.0041791125505887617) -- (0.0019151284697052121, 0.00084678901330169278);

\end{tikzpicture}

%% file: sudoku-tikz.tex
\begin{tabular}{cccc} 
 \begin{tikzpicture}[scale=0.25]
  \draw (0,0,0) -- (0,0,9) -- (9,0,9) -- (9,0,0) -- cycle;
  \foreach \i  in {3,6} {
    \draw (\i,0,0) -- (\i,0,9);
    \draw (0,0,\i) -- (9,0,\i);
  }
  \foreach \i  in {1,2,4,5,7,8} {
    \draw[dotted] (\i,0,0) -- (\i,0,9);
    \draw[dotted] (0,0,\i) -- (9,0,\i);
  }
  \foreach \i in {0,3,6} {
    \draw[color=black!50] (0,0,\i) -- (0,9,\i);
    \draw[color=black!50] (\i,0,0) -- (\i,9,0);
    \draw[color=black!50] (9,\i,0) -- (0,\i,0) -- (0,\i,9);
  }
  \foreach \i in {0,3,6,9} { \foreach \j in {9} { \draw (\i,0,9) -- (\i,\j,9) -- (\i,\j,0); } }
  \foreach \i in {0,3,6,9} { \foreach \j in {9} { \draw (9,0,\i) -- (9,\j,\i) -- (0,\j,\i); } }
  \foreach \i in {3,6} { \draw (9,\i,0) -- (9,\i,9) -- (0,\i,9); }
  \draw (0,9,0) -- (9,9,0) -- (9,9,9) -- (0,9,9) -- cycle;
  \draw[fill=green, fill opacity=0.3] (9,1,9) -- (9,1,8) -- (0,1,8) -- (0,1,9) -- cycle;
  \draw[fill=green, fill opacity=0.3] (9,0,9) -- (9,1,9) -- (0,1,9) -- (0,0,9) -- cycle;
  \draw[fill=green, fill opacity=0.3] (9,0,9) -- (9,1,9) -- (9,1,8) -- (9,0,8) -- cycle;
\end{tikzpicture}  
 &
 \begin{tikzpicture}[scale=0.25]
  \draw (0,0,0) -- (0,0,9) -- (9,0,9) -- (9,0,0) -- cycle;
  \foreach \i  in {3,6} {
    \draw (\i,0,0) -- (\i,0,9);
   \draw (0,0,\i) -- (9,0,\i);
  }
  \foreach \i  in {1,2,4,5,7,8} {
    \draw[dotted] (\i,0,0) -- (\i,0,9);
    \draw[dotted] (0,0,\i) -- (9,0,\i);
  }
  \foreach \i in {0,3,6} {
    \draw[color=black!50] (0,0,\i) -- (0,9,\i);
    \draw[color=black!50] (\i,0,0) -- (\i,9,0);
    \draw[color=black!50] (9,\i,0) -- (0,\i,0) -- (0,\i,9);
  }
  \foreach \i in {0,3,6,9} { \foreach \j in {9} { \draw (\i,0,9) -- (\i,\j,9) -- (\i,\j,0); } }
  \foreach \i in {0,3,6,9} { \foreach \j in {9} { \draw (9,0,\i) -- (9,\j,\i) -- (0,\j,\i); } }
  \foreach \i in {3,6} { \draw (9,\i,0) -- (9,\i,9) -- (0,\i,9); }
  \draw (0,9,0) -- (9,9,0) -- (9,9,9) -- (0,9,9) -- cycle;
  \draw[fill=green, fill opacity=0.3] (9,1,9) -- (8,1,9) -- (8,1,0) -- (9,1,0) -- cycle;
  \draw[fill=green, fill opacity=0.3] (9,0,9) -- (9,1,9) -- (9,1,0) -- (9,0,0) -- cycle;
  \draw[fill=green, fill opacity=0.3] (9,0,9) -- (9,1,9) -- (8,1,9) -- (8,0,9) -- cycle;
\end{tikzpicture}
& 
 \begin{tikzpicture}[scale=0.25] 

  \draw (0,0,0) -- (0,0,9) -- (9,0,9) -- (9,0,0) -- cycle;
  \foreach \i  in {3,6} {
    \draw (\i,0,0) -- (\i,0,9);
    \draw (0,0,\i) -- (9,0,\i);
  }
  \foreach \i  in {1,2,4,5,7,8} {
    \draw[dotted] (\i,0,0) -- (\i,0,9);
    \draw[dotted] (0,0,\i) -- (9,0,\i);
  }
  \foreach \i in {0,3,6} {
    \draw[color=black!50] (0,0,\i) -- (0,9,\i);
    \draw[color=black!50] (\i,0,0) -- (\i,9,0);
    \draw[color=black!50] (9,\i,0) -- (0,\i,0) -- (0,\i,9);
  }
  \foreach \i in {0,3,6,9} { \foreach \j in {9} { \draw (\i,0,9) -- (\i,\j,9) -- (\i,\j,0); } }
  \foreach \i in {0,3,6,9} { \foreach \j in {9} { \draw (9,0,\i) -- (9,\j,\i) -- (0,\j,\i); } }
  \foreach \i in {3,6} { \draw (9,\i,0) -- (9,\i,9) -- (0,\i,9); }
  \draw (0,9,0) -- (9,9,0) -- (9,9,9) -- (0,9,9) -- cycle;
  \draw[fill=green, fill opacity=0.3] (9,0,9) -- (9,1,9) -- (9,1,6) -- (9,0,6) -- cycle;
  \draw[fill=green, fill opacity=0.3] (9,0,9) -- (9,1,9) -- (6,1,9) -- (6,0,9) -- cycle;
  \draw[fill=green, fill opacity=0.3] (9,1,9) -- (6,1,9) -- (6,1,6) -- (9,1,6) -- cycle;
\end{tikzpicture} 
 &
  \begin{tikzpicture}[scale=0.25] 
  \draw (0,0,0) -- (0,0,9) -- (9,0,9) -- (9,0,0) -- cycle;
  \foreach \i  in {3,6} {
    \draw (\i,0,0) -- (\i,0,9);
    \draw (0,0,\i) -- (9,0,\i);
  }
  \foreach \i  in {1,2,4,5,7,8} {
    \draw[dotted] (\i,0,0) -- (\i,0,9);
    \draw[dotted] (0,0,\i) -- (9,0,\i);
  }
  \foreach \i in {0,3,6} {
    \draw[color=black!50] (0,0,\i) -- (0,9,\i);
    \draw[color=black!50] (\i,0,0) -- (\i,9,0);
    \draw[color=black!50] (9,\i,0) -- (0,\i,0) -- (0,\i,9);
  }
  \foreach \i in {0,3,6,9} { \foreach \j in {9} { \draw (\i,0,9) -- (\i,\j,9) -- (\i,\j,0); } }
  \foreach \i in {0,3,6,9} { \foreach \j in {9} { \draw (9,0,\i) -- (9,\j,\i) -- (0,\j,\i); } }
  \foreach \i in {3,6} { \draw (9,\i,0) -- (9,\i,9) -- (0,\i,9); }
  \draw (0,9,0) -- (9,9,0) -- (9,9,9) -- (0,9,9) -- cycle;
  \draw[fill=green, fill opacity=0.3] (9,0,9) -- (9,0,8) -- (9,9,8) -- (9,9,9) -- cycle;
  \draw[fill=green, fill opacity=0.3] (9,0,9) -- (8,0,9) -- (8,9,9) -- (9,9,9) -- cycle;
  \draw[fill=green, fill opacity=0.3] (9,9,9) -- (8,9,9) -- (8,9,8) -- (9,9,8) -- cycle;
\end{tikzpicture}
 
 \\ $C_1\qquad$ & $C_2\qquad$ & $C_3\qquad$ & $C_5\qquad$ \\
\end{tabular}

%% file: camel_0.tex
\bf
\tabcolsep=0.05cm
\renewcommand{\arraystretch}{0.15}
\begin{tabular}{*{2}{|>{\centering\arraybackslash}m{2.5ex}}|m{2.5ex}*{10}{|>{\centering\arraybackslash}m{2.5ex}}|} \cline{4-13}
\multicolumn{3}{c|}{}   &     &     &     &     &     &     &  1  &     &     &    \\ [2.5ex] \cline{4-13}
\multicolumn{3}{c|}{}   &     &     &     &  2  &     &     &  4  &  1  &  2  &  \raisebox{-1.85ex}{2} \\ [2.5ex] \cline{4-13}
\multicolumn{3}{c|}{}   &  2  &  3  &  1  &  1  &  5  &  4  &  1  &  5  &  2  &  \raisebox{-1.85ex}{1} \\ [2.5ex] \cline{4-13}
\noalign{\vspace{5ex}}                                                                         \cline{1-2} \cline{4-13}
1 & 2                  &&     &     &     &     &     &     &     &     &     &     \\ [2.5ex] \cline{1-2} \cline{4-13}
  & 2                  &&     &     &     &     &     &     &     &     &     &     \\ [2.5ex] \cline{1-2} \cline{4-13}
  & 1                  &&     &     &     &     &     &     &     &     &     &     \\ [2.5ex] \cline{1-2} \cline{4-13}
  & 1                  &&     &     &     &     &     &     &     &     &     &     \\ [2.5ex] \cline{1-2} \cline{4-13}
  & 2                  &&     &     &     &     &     &     &     &     &     &     \\ [2.5ex] \cline{1-2} \cline{4-13}
2 & 4                  &&     &     &     &     &     &     &     &     &     &     \\ [2.5ex] \cline{1-2} \cline{4-13}
2 & 6                  &&     &     &     &     &     &     &     &     &     &     \\ [2.5ex] \cline{1-2} \cline{4-13}
  & 8                  &&     &     &     &     &     &     &     &     &     &     \\ [2.5ex] \cline{1-2} \cline{4-13}
1 & 1                  &&     &     &     &     &     &     &     &     &     &     \\ [2.5ex] \cline{1-2} \cline{4-13}
2 & 2                  &&     &     &     &     &     &     &     &     &     &     \\ [2.5ex] \cline{1-2} \cline{4-13}
\end{tabular}


%% file: camel_1.tex
\begin{tabular}{*{10}{|>{\centering\arraybackslash}m{2.5ex}}|}  \hline
     &     & \cb &     &     &     & \cb & \cb &     &     \\ [2.5ex] \hline
     &     &     &     &     &     & \cb & \cb &     &     \\ [2.5ex] \hline
     &     &     &     &     &     &     & \cb &     &     \\ [2.5ex] \hline
     &     &     &     &     &     &     &     & \cb &     \\ [2.5ex] \hline
     &     &     &     &     & \cb & \cb &     &     &     \\ [2.5ex] \hline
 \cb & \cb &     &     & \cb & \cb & \cb & \cb &     &     \\ [2.5ex] \hline
     & \cb & \cb &     & \cb & \cb & \cb & \cb & \cb & \cb \\ [2.5ex] \hline
     &     & \cb & \cb & \cb & \cb & \cb & \cb & \cb & \cb \\ [2.5ex] \hline
     &     &     &     & \cb &     &     & \cb &     &     \\ [2.5ex] \hline
     &     & \cb & \cb &     &     & \cb & \cb &     &     \\ [2.5ex] \hline
 \hline
\end{tabular}

%% file: camel_2.tex
\begin{tabular}{*{10}{|>{\centering\arraybackslash}m{2.5ex}}|}  \hline
     &     &     & \cb &     &     &     &     & \cb & \cb \\ [2.5ex] \hline
     &     &     &     &     &     & \cb & \cb &     &     \\ [2.5ex] \hline
     &     &     &     &     &     &     &     & \cb &     \\ [2.5ex] \hline
     &     &     &     &     &     &     &     & \cb &     \\ [2.5ex] \hline
     &     &     &     &     &     &     &     & \cb & \cb \\ [2.5ex] \hline
 \cb & \cb &     &     & \cb & \cb & \cb & \cb &     &     \\ [2.5ex] \hline
 \cb & \cb &     & \cb & \cb & \cb & \cb & \cb & \cb &     \\ [2.5ex] \hline
 \cb & \cb & \cb & \cb & \cb & \cb & \cb & \cb &     &     \\ [2.5ex] \hline
     &     &     &     & \cb &     &     & \cb &     &     \\ [2.5ex] \hline
     &     &     & \cb & \cb &     & \cb & \cb &     &     \\ [2.5ex] \hline
 \hline
\end{tabular}

%% file: camel_3.tex
\begin{tabular}{*{10}{|>{\centering\arraybackslash}m{2.5ex}}|}  \hline
 \cb &     &     &     &     & \cb & \cb &     &     &     \\ [2.5ex] \hline
     &     &     &     &     &     &     &     & \cb & \cb \\ [2.5ex] \hline
     &     &     &     &     &     &     &     & \cb &     \\ [2.5ex] \hline
     &     &     &     &     &     &     &     & \cb &     \\ [2.5ex] \hline
     &     &     &     &     & \cb & \cb &     &     &     \\ [2.5ex] \hline
     & \cb & \cb &     & \cb & \cb & \cb & \cb &     &     \\ [2.5ex] \hline
 \cb & \cb &     & \cb & \cb & \cb & \cb & \cb & \cb &     \\ [2.5ex] \hline
 \cb & \cb & \cb & \cb & \cb & \cb & \cb & \cb &     &     \\ [2.5ex] \hline
     &     &     &     & \cb &     &     & \cb &     &     \\ [2.5ex] \hline
     &     &     & \cb & \cb &     & \cb & \cb &     &     \\ [2.5ex] \hline
 \hline
\end{tabular}

%% file: camel_4.tex
\begin{tabular}{*{10}{|>{\centering\arraybackslash}m{2.5ex}}|}  \hline
     &     &     &     &     &     & \cb &     & \cb & \cb \\ [2.5ex] \hline
     &     &     &     &     &     &     &     & \cb & \cb \\ [2.5ex] \hline
     &     &     &     &     &     &     &     & \cb &     \\ [2.5ex] \hline
     & \cb &     &     &     &     &     &     &     &     \\ [2.5ex] \hline
     &     &     &     &     & \cb & \cb &     &     &     \\ [2.5ex] \hline
 \cb & \cb &     &     & \cb & \cb & \cb & \cb &     &     \\ [2.5ex] \hline
 \cb & \cb &     & \cb & \cb & \cb & \cb & \cb & \cb &     \\ [2.5ex] \hline
     & \cb & \cb & \cb & \cb & \cb & \cb & \cb & \cb &     \\ [2.5ex] \hline
     &     &     &     & \cb &     &     & \cb &     &     \\ [2.5ex] \hline
     &     &     & \cb & \cb &     & \cb & \cb &     &     \\ [2.5ex] \hline
 \hline
\end{tabular}

%% file: camel_5.tex
\begin{tabular}{*{10}{|>{\centering\arraybackslash}m{2.5ex}}|}  \hline
     &     &     &     &     &     & \cb &     & \cb & \cb \\ [2.5ex] \hline
     &     &     &     &     &     &     &     & \cb & \cb \\ [2.5ex] \hline
     &     &     &     &     &     &     & \cb &     &     \\ [2.5ex] \hline
     &     &     &     &     &     & \cb &     &     &     \\ [2.5ex] \hline
     &     &     &     &     & \cb & \cb &     &     &     \\ [2.5ex] \hline
 \cb & \cb &     &     & \cb & \cb & \cb & \cb &     &     \\ [2.5ex] \hline
 \cb & \cb &     & \cb & \cb & \cb & \cb & \cb & \cb &     \\ [2.5ex] \hline
     & \cb & \cb & \cb & \cb & \cb & \cb & \cb & \cb &     \\ [2.5ex] \hline
     &     &     &     & \cb &     &     & \cb &     &     \\ [2.5ex] \hline
     &     &     & \cb & \cb &     & \cb & \cb &     &     \\ [2.5ex] \hline
 \hline
\end{tabular}

%% file: camel_6.tex
\begin{tabular}{*{10}{|>{\centering\arraybackslash}m{2.5ex}}|}  \hline
     &     &     &     &     &     & \cb &     & \cb & \cb \\ [2.5ex] \hline
     &     &     &     &     &     &     &     & \cb & \cb \\ [2.5ex] \hline
     &     &     &     &     &     &     & \cb &     &     \\ [2.5ex] \hline
     &     &     &     &     &     &     &     &     & \cb \\ [2.5ex] \hline
     &     &     &     &     & \cb & \cb &     &     &     \\ [2.5ex] \hline
 \cb & \cb &     &     & \cb & \cb & \cb & \cb &     &     \\ [2.5ex] \hline
 \cb & \cb &     & \cb & \cb & \cb & \cb & \cb & \cb &     \\ [2.5ex] \hline
     & \cb & \cb & \cb & \cb & \cb & \cb & \cb & \cb &     \\ [2.5ex] \hline
     &     &     &     & \cb &     &     & \cb &     &     \\ [2.5ex] \hline
     &     &     & \cb & \cb &     & \cb & \cb &     &     \\ [2.5ex] \hline
 \hline
\end{tabular}